\documentclass[a4paper,11pt]{article}

\usepackage{cmap}

\usepackage[usenames,dvipsnames]{xcolor}
\definecolor{myblue}{rgb}{0.21, 0.34, 0.74}
\definecolor{mygrey}{rgb}{0.55, 0.57, 0.67}
\definecolor{myred}{rgb}{0.79, 0.0, 0.09}
\definecolor{JuliaRed}{RGB}{204, 52, 51}

\usepackage[pdftex, colorlinks=true, bookmarksopen=true, linkcolor=blue, citecolor=blue]{hyperref}

\usepackage[T1]{fontenc}
\usepackage[utf8]{inputenc}

\usepackage{lmodern}
\usepackage{libertine}

\usepackage{bbm}
\usepackage{mathrsfs}
\usepackage{amssymb}
\usepackage{amsmath}
\usepackage{upgreek}
\usepackage{mathtools}

\usepackage{amsthm}
\usepackage[nottoc,notlot,notlof]{tocbibind}

\usepackage{comment}
\usepackage[font=small,labelfont=bf]{caption}

\usepackage{authblk}

\usepackage{subcaption}

\usepackage[svgnames,pdf]{pstricks}

\usepackage{wrapfig}
\usepackage{upgreek}
\usepackage{bm}
\usepackage[scr=boondoxo]{mathalfa}
\usepackage{comment}
\usepackage[font=small,labelfont=bf]{caption}

\usepackage[noabbrev,capitalize,nameinlink]{cleveref}

\usepackage{tikz}
\usepackage{bbm}
\usepackage[T1]{fontenc}

\usetikzlibrary{arrows.meta}

\usepackage{environ}
\usepackage{framed}
\usepackage{url}
\usepackage[labelfont=bf]{caption}
\usepackage{cite}
\usepackage{framed}
\usepackage[framemethod=tikz]{mdframed}
\usepackage{appendix}
\usepackage{graphicx}
\usepackage[textsize=tiny]{todonotes}
\usepackage{tcolorbox}
\usepackage{multicol}
\usepackage{enumerate}
\allowdisplaybreaks[1]
\usepackage{enumerate}
\usepackage{stmaryrd}

\usepackage{upgreek}

\usepackage[shortlabels]{enumitem}
\crefformat{enumi}{#2#1#3}
\crefrangeformat{enumi}{#3#1#4 to~#5#2#6}
\crefmultiformat{enumi}{#2#1#3}
{ and~#2#1#3}{, #2#1#3}{ and~#2#1#3}

\DeclareSymbolFont{symbolsC}{U}{pxsyc}{m}{n}
\SetSymbolFont{symbolsC}{bold}{U}{pxsyc}{bx}{n}
\DeclareFontSubstitution{U}{pxsyc}{m}{n}
\DeclareMathSymbol{\medcircle}{\mathbin}{symbolsC}{7}

\crefname{equation}{}{} 
\AtBeginEnvironment{appendices}{\crefalias{section}{appendix}} 

\numberwithin{equation}{section}
\newtheorem{theorem}{Theorem}[section]
\newtheorem{proposition}[theorem]{Proposition}
\newtheorem{lemma}[theorem]{Lemma}

\newtheorem{remark}[theorem]{Remark}

\newtheorem{corollary}[theorem]{Corollary}

\newtheorem*{question*}{Question}

\theoremstyle{definition}

\newtheorem*{definition*}{Definition}

\theoremstyle{remark}

\newcommand{\mb}{\mathbb}
\newcommand{\mbf}{\mathbf}

\newcommand{\msc}{\mathscr}

\newcommand{\on}{\operatorname}
\newcommand{\wh}{\widehat}

\newcommand{\edit}[1]{{\color{black}#1}}
\newcommand{\diff}{\, \mathrm{d}}

\renewcommand{\le}{\leqslant}
\renewcommand{\ge}{\geqslant}
\renewcommand{\leq}{\leqslant}
\renewcommand{\geq}{\geqslant}

\let\originalleft\left
\let\originalright\right
\renewcommand{\left}{\mathopen{}\mathclose\bgroup\originalleft}
\renewcommand{\right}{\aftergroup\egroup\originalright}

\allowdisplaybreaks

\newif\ifpublic
\publictrue

\ifpublic

\newcommand{\ignore}[1]{}

\else

\fi

\definecolor{MIT}{cmyk}{.24, 1.00, .78, .17}

\title{On the number of modes of Gaussian kernel density estimators}

\author{Borjan Geshkovski}
\affil{Inria \& Sorbonne Université}
\author{Philippe Rigollet}
\affil{MIT}
\author{Yihang Sun}
\affil{Stanford University}

\date{ \today }

\begin{document}
\setlist[itemize,enumerate]{left=0pt}

%
%

\maketitle

%
%

\begin{abstract}

We consider the Gaussian kernel density estimator with bandwidth $\beta^{-\frac12}$ of $n$ iid Gaussian samples. Using the Kac-Rice formula and an Edgeworth expansion, we prove that the expected number of modes on the real line scales as $\Theta(\sqrt{\beta\log\beta})$ as $\beta,n\to\infty$ provided $n^c\lesssim \beta\lesssim n^{2-c}$ for some constant $c>0$. An impetus behind this investigation is to determine the number of clusters to which Transformers are drawn in a metastable state.
			
\bigskip


\medskip

\noindent \textbf{\textsc{ams} classification.}\quad \textsc{62G07, 60G60, 60F05, 68T07}.
\end{abstract}
	
\thispagestyle{empty}

\setcounter{tocdepth}{2}
\tableofcontents

%
%

\section{Introduction}

\subsection{Setup and main result}
\label{sec:setup}

For $\beta >0$ and $X_1, \dots, X_n \stackrel{\text{iid}}{\sim} N(0, 1)$, the \emph{Gaussian kernel density estimator (KDE)} with bandwidth $h=\beta^{-\frac12}$ is defined as
\begin{equation}\label{eq:gkde}
\wh{P}_n(t)\coloneqq \frac{1}{n}\sum_{i=1}^n \mathsf{K}_{h}\ast\delta_{X_i}(t) =  \frac{\sqrt{\beta}}{n\sqrt{2\pi}}\sum_{i=1}^n e^{-\frac{\beta}{2}(t-X_i)^2}, \hspace{1cm} t\in\mb{R}.
\end{equation}	
Here, ``Gaussian'' refers to the choice of kernel $\mathsf{K}_h$.

In this paper we are interested in determining the expected number of modes (local maxima) of $\wh{P}_n$ over $\mb{R}$. 
While this is a classical question, addressed in even more general settings  than \eqref{eq:gkde}---such as non-Gaussian kernels, compactly supported samples, and higher dimensions \cite{mammen91,mammen95,mammen97}---a definite answer has not been given in the literature. 
Indeed, the best-known results fall into one of two settings: either considering samples drawn from a compactly supported density  (instead of $N(0,1)$ as done here), or counting the modes within a fixed compact interval. 
In the special case of the Gaussian KDE \eqref{eq:gkde}, one has in the latter setting for instance

\begin{theorem}[{\!\cite[Thm.~1]{mammen95}}]
\label{thm:mammen}
Let $\wh{P}_n$ be the Gaussian KDE defined in \eqref{eq:gkde}, with bandwidth $h\coloneqq\beta^{-\frac12}>0$, of $X_1, \dots, X_n \stackrel{\text{iid}}{\sim} N(0, 1)$. 
Asymptotically as $n\to\infty$, the expected number $N$ of modes of $\wh{P}_n$ in a fixed interval $[a, b]$ is 
\begin{itemize}\edit{
    \item $\mbf{1}\{0\in [a, b]\}+o(1)$ if $\beta\ll n^{\frac25}$,
    \item $\Theta(1)$ if $\beta\asymp n^{\frac25}$
    \item $\Theta\left(n^{-\frac12}\beta^{\frac54}\right)=o\left(\sqrt{\beta}\right)$ if $n^{\frac25}\ll \beta\ll n^{\frac23}$,
    \item and $\Theta\left(\sqrt{\beta}\right)$ if $n^{\frac23}\lesssim \beta \ll n^2/\log^6 n$.}
\end{itemize}
\end{theorem}

In \cite{mammen91, mammen95, mammen97}, the authors additionally conduct more refined casework on the bandwidth to provide more precise estimates, such as pinpointing the leading constants. In fact, \cite{mammen91} \emph{does} count modes in \(\mathbb{R}\), but the underlying distribution of the samples \(X_i\) is supported on a closed interval (thus excluding $N(0, 1)$), so there are no modes outside the interval anyway.

\edit{
In the case of counting modes of \eqref{eq:gkde} over $\mb{R}$, \cite{carreira2003number} provides an upper bound of $n$ using scale-space theory by showing adding components one-by-one to a Gaussian mixture increases the mode-count by at most one each time. Our main result velow provides a precise answer in this case.
For the sake of clarity, we stick to the regime where 
\[2\log n-\log \beta \asymp \log\beta\asymp \log n.\] 
Through refined computations, one can determine the modes in the regime $1 \ll \beta \lesssim n^2/\log^{\Theta(1)}(n)$ and also pinpoint the leading constant. We state our main theorem, and will comment on how to do expand the regime in appropriate places. 
}
\begin{theorem}\label{thm:main-result}
Let $\wh{P}_n$ be the Gaussian KDE defined in \eqref{eq:gkde}, with bandwidth $\beta^{-\frac12}$, of $X_1, \dots, X_n \stackrel{\text{iid}}{\sim} N(0, 1)$. Suppose $n^c\lesssim \beta \lesssim n^{2-c}$ for arbitrarily small $c>0$.
Then asymptotically as $n,\beta\to\infty$,
\begin{enumerate}
    \item In expectation over $X_i$, the number of modes of $\wh{P}_n$ is $\Theta\left( \sqrt{\beta\log \beta} \right)$. 
\item Almost all modes lie in two intervals of length $\Theta\left(\sqrt{\log\beta}\right)$---namely, the expected number of modes $t\in\mb{R}$, such that $t^2\not\in \left[2\log n-3\log\beta,2\log n-\log\beta\right]$, is $o\left(\sqrt{\beta\log\beta}\right)$.
\end{enumerate}
\end{theorem}

\begin{figure}
    \centering
    \includegraphics[scale=0.24]{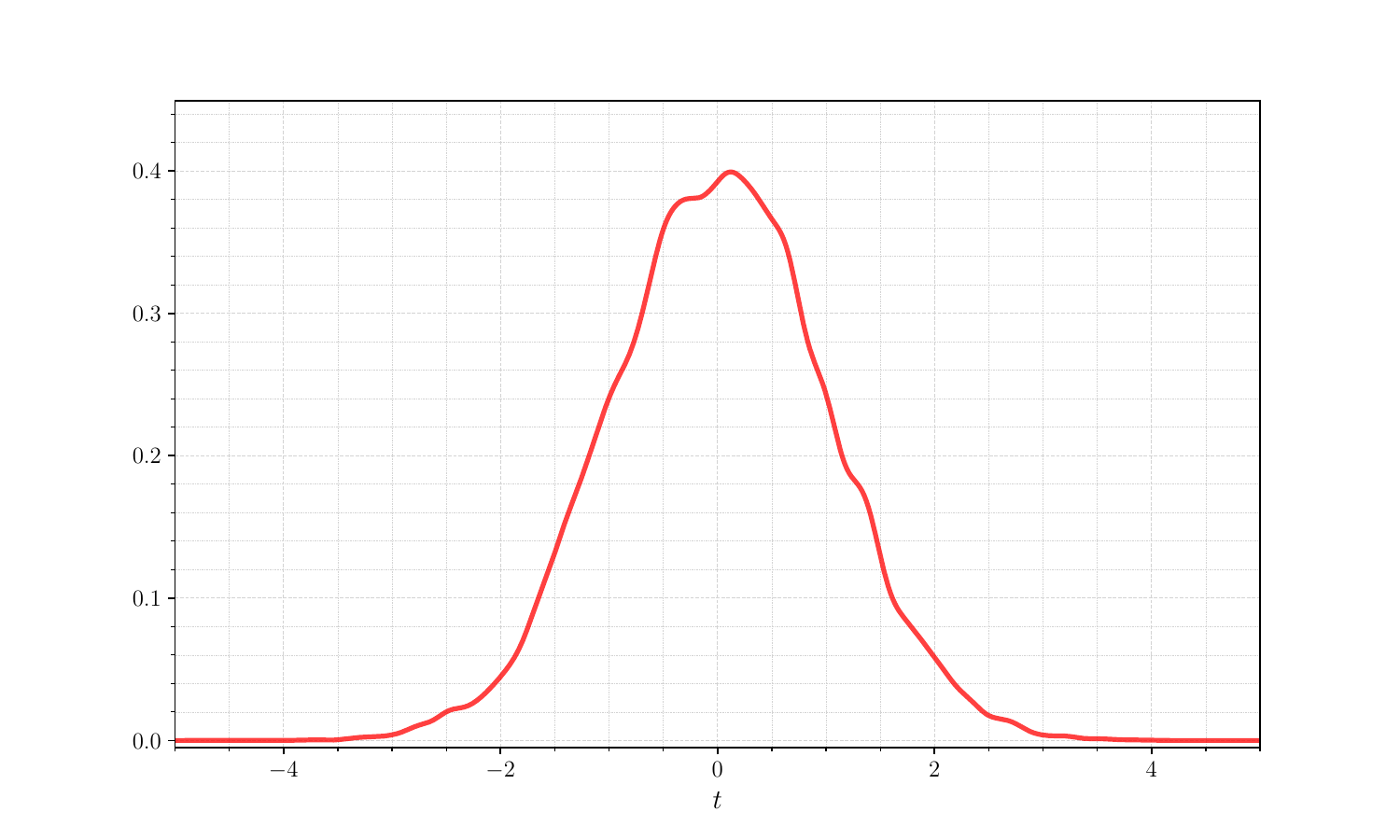}
    \includegraphics[scale=0.24]{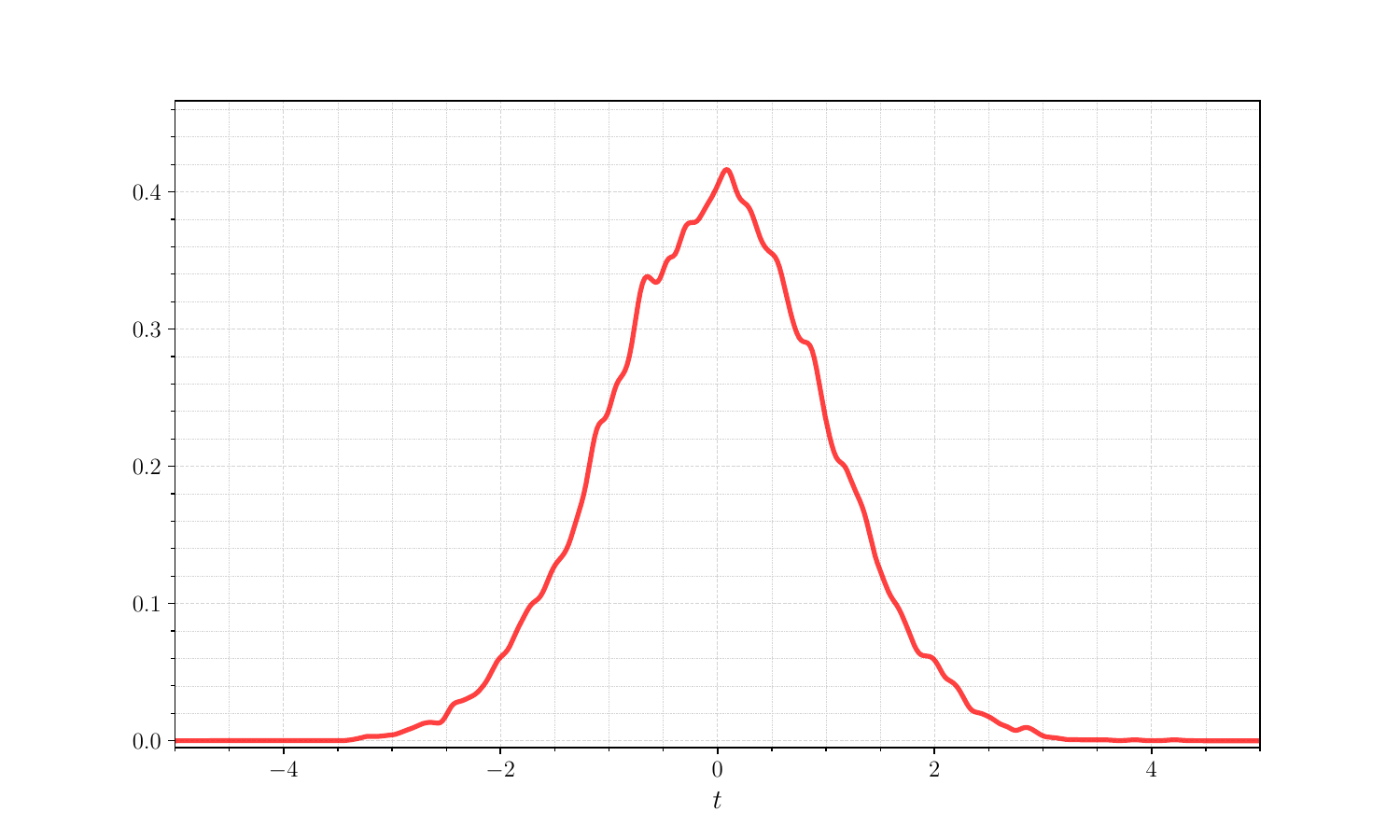}
    \caption{\edit{A realization of the kernel density estimator $\widehat P_n$ in \eqref{eq:gkde} for $n=10^4$, with $\beta=100$ (left) and $\beta=300$ (right). Larger $\beta$ narrows the Gaussian kernel, which sharpens $\widehat P_n$ and reveals more small peaks on the shoulders, while the central peak remains single. \Cref{thm:main-result} later quantifies where and how many such peaks appear.}}
    \label{fig: kde}
\end{figure}

\edit{In fact, we can bound on the rate of convergence of the little-$o$ in Point 2. This is spelled-out in \cref{prop:main-int,prop:main-tail}.} Several comments are in order.

\begin{remark}
    \label{rem:minimax}
    \begin{itemize}
        \item To better appreciate the range of values for $\beta$ in this theorem as well as subsequent ones, we use minimax theory as a benchmark; see, e.g.,~\cite{Tsy09}.  The reparametrization $h=\beta^{-\frac12}$ is motivated by the connection to the Transformer model described in \Cref{sec:transformers}. Using an optimal bias-variance tradeoff~\cite[Chapter~1]{Tsy09}\footnote{\edit{With $h=\beta^{-\frac12}$, the usual bias–variance calculus for $s$-smooth densities gives $h\asymp n^{-\frac{1}{2s+1}}$ and hence $\beta\asymp n^{\frac{2}{2s+1}}$ ~\cite[Ch. 1]{Tsy09}}.}, we see that the optimal scaling of the bandwidth parameter $h$ depends on the smoothness of the underlying density of interest: if the underlying density has $s$ bounded (fractional) derivatives, then the optimal choice of $h$ is given by $h \asymp n^{-\frac{1}{2s+1}}$. This gives $\beta \asymp n^\frac{2}{2s+1}$. For $s>0$, we get $\beta \in [n^{c}, n^{2-c}]$ for some $c>0$. In particular, the transition of the number of modes from 1 to $\sqrt{\beta}$ in \Cref{thm:mammen} is achieved for $\beta \approx  n^{\frac23}$, which is the optimal choice for Lipschitz densities. The  message of our main Theorem~\ref{thm:main-result} above is that this scaling in $\sqrt{\beta}$ is the prevailing one for the whole range $\beta \in [n^{c}, n^{2-c}]$ if one does not restrict counting modes in a bounded interval $[a,b]$.
        \item Point 2. in \Cref{thm:main-result} shows that most of the modes are at distance at least $C\log n$ from the origin provided $\beta > n^{\frac{2-C}{3}}$ for $C>0$ small. This corresponds to a choice of a bandwidth adapted to smoothness $s<1$. This result is in agreement with and completes the picture drawn by \Cref{thm:mammen}.
    \end{itemize}
\end{remark}

\begin{remark}
\label{rmk:empirical}
We further motivate Point 2. in \cref{thm:main-result} by considering a qualitative picture of the distribution of the modes displayed in \Cref{figure:approx}. 
\begin{itemize}
    \item Near the origin, we find most of the samples $X_i$ and they are densely packed in the shape of a Gaussian. The corresponding Gaussian summands in \cref{eq:gkde} cancel to create one mode, as shown already in \cref{thm:mammen}.
    
    \item In the two intervals of length $\Theta\left(\sqrt{\log \beta}\right)$, the samples $X_i$ are separated enough that the corresponding Gaussian summands do not cancel, but rather form $\Omega\left(\sqrt{\log\beta}\right)$ \emph{isolated bumps}, as discussed in more generality in \cite[Section 9.3]{devgyo}. 
    
    \item Further away at the tails, the phenomena of isolated bumps occur, but there are so few samples $X_i$ that the number of modes created is a negligible fraction.
    \end{itemize}

\edit{Write $\widehat P_n'(t)=\mathbb E\,\widehat P_n'(t)+
(\widehat P_n'(t)-\mathbb E\widehat P_n'(t))$.
The first term (“bias’’) reflects the deterministic drift toward a single broad mode, while the second 
(“variance’’) creates random sign changes that generate extra modes. For the rescaled field 
$F_n(t)=-c\,\widehat P_n'(t)$ in \eqref{eq:Fn}, \Cref{lem:moments-p} yields 
\[
\mathrm{SNR}(t)^2\coloneqq
\frac{|\mathbb E F_n(t)|^2}{\mathrm{Var}F_n(t)}
\asymp n t^2\beta^{-\frac32} e^{-\frac{t^2}{2}}.
\]
When $\mathrm{SNR}(t)\gg1$ the bias dominates and no additional modes appear; when 
$\mathrm{SNR}(t)\ll1$ the variance takes over and modes proliferate. The crossover 
$\mathrm{SNR}(t)\approx1$ gives the inner edge $t^2\approx 2\log n-3\log\beta$ (up to $\log t$ terms). 
On the other hand, to form isolated bumps we also need at least one point in a kernel window of 
width $h=\beta^{-\frac12}$, i.e. $n\varphi(t)h\approx 1$, which gives the outer edge 
$t^2\approx 2\log n-\log\beta$. Together these two thresholds select the belt 
$t^2\in[2\log n-3\log\beta,2\log n-\log\beta]$ and, in particular, center the localization at 
$|t|\approx\sqrt{2\log n-\log\beta}$.}

We revisit this discussion and \Cref{figure:approx} in \cref{rmk:delicate}.
\end{remark}

\edit{
\begin{remark}[Belt width] From \Cref{sec: 2.3} the Kac–Rice density is proportional to
$\sqrt{\beta}e^{-A_t}$ with $A_t \asymp \beta^{-\frac32}n t^2 e^{-\frac{t^2}{2}}$, so the mass concentrates
where $A_t=O(1)$. This pins down $t^2 = 2\log n - c\log\beta + O(\log\log\beta)$ with
$c\in\{1,3\}$—precisely the endpoints of \Cref{thm:main-result}—and converting from $t^2$ to $t$ turns the
$2\log\beta$ gap into a belt of length $\Delta t \approx (2t_*)^{-1}\cdot 2\log\beta \asymp \sqrt{\log\beta}$
around $t_*\approx\sqrt{2\log n - \log\beta}$.
\end{remark}
}

\edit{\begin{remark}
We compare our result with \cref{thm:mammen}.
Let $J$ be the union of the symmetric intervals of length $\Theta(\sqrt{\log\beta})$ in Point 2, i.e. the ``two belts’’ region where almost all modes lie. Our proof will show that the density of modes is $\Theta(\sqrt{\beta})$ whenever $t\in J$, and $o(\sqrt{\beta})$ whenever $t\not\in J$. This gives the \cref{thm:main-result} upon integrating over $t$, and explains the threshold $\beta\asymp n^{\frac23}$ in \cref{thm:mammen}: 
\begin{itemize}
\item If $\beta\gtrsim n^{\frac23}$, then $[a, b]\subset J$ for sufficiently large $n$ and $\beta$, so the density of modes is $\Theta(\sqrt{\beta})$ everywhere on $[a, b]$, giving $\Theta(\sqrt{\beta})$ modes in total.
\item If $\beta\ll n^{\frac23}$, then $[a, b]$ is between (and outside of) the two belts of $J$ for sufficiently large $n$ and $\beta$, so the density of modes is $o(\sqrt{\beta})$ everywhere on $[a, b]$, giving $o(\sqrt{\beta})$ modes in total. In fact, \cref{prop:main-int} shows this symmetric interval $T'$ between the two belts has $O(\sqrt{\beta})$ modes in total. We can see that $T'$ has length $\omega_{n\to\infty}(1)$ and the mode density is increasing as we move away from $0$, so the number of modes in $[a, b]$ must be a $o(1)$-fraction of the modes in $T'$, i.e. it is $o(\sqrt{\beta})$.
\end{itemize}
Hence, this corollary of our result implies the last two bullet points of \cref{thm:mammen}. Similarly, by truncating to more refined intervals separated by $t^2 = 2\log n - 5\log \beta$, we can hope to recover the threshold $\beta\asymp n^{\frac25}$ given in the first three bullet points of \cref{thm:mammen}, but we do not pursue this here.
\end{remark}}

\begin{figure}[!ht]
    \centering
    \includegraphics[scale=0.285]{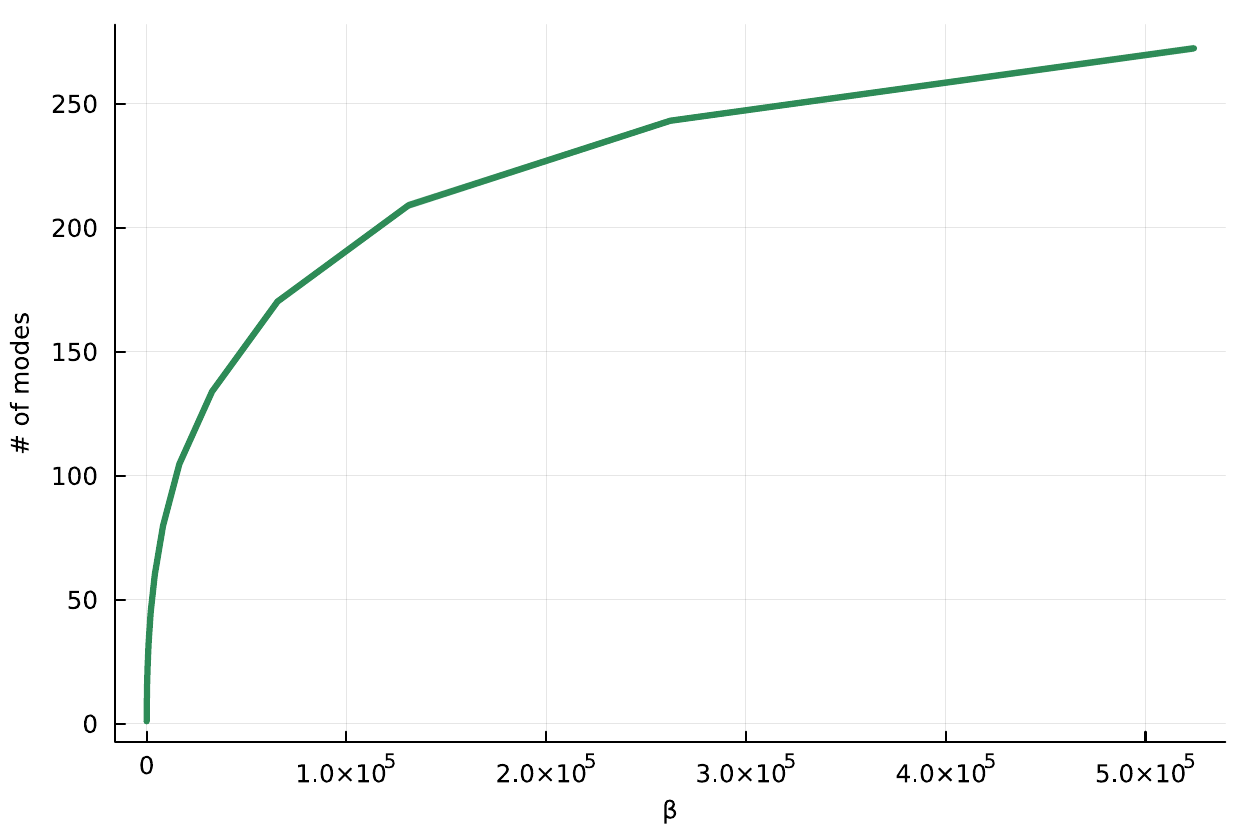}
    \hspace{0.1cm}
    \includegraphics[scale=0.285]{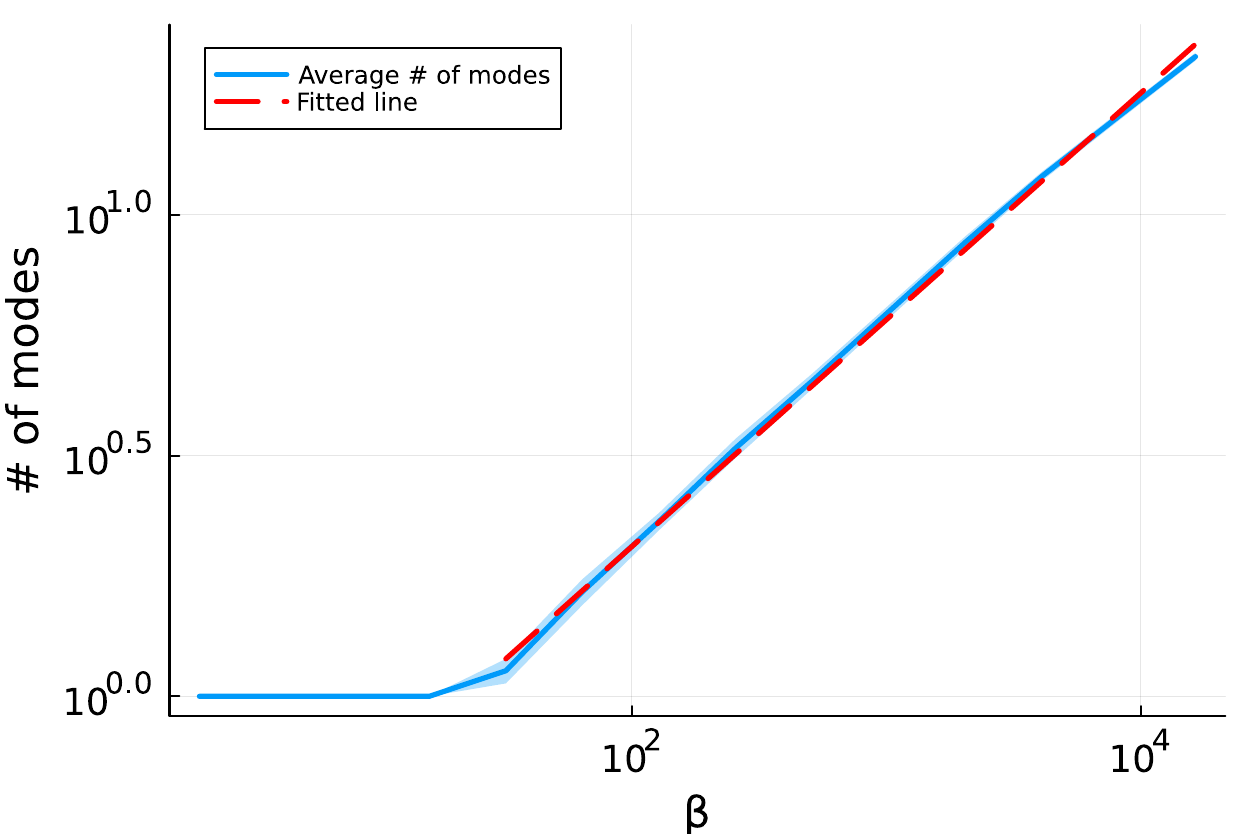}
    \includegraphics[scale=0.285]{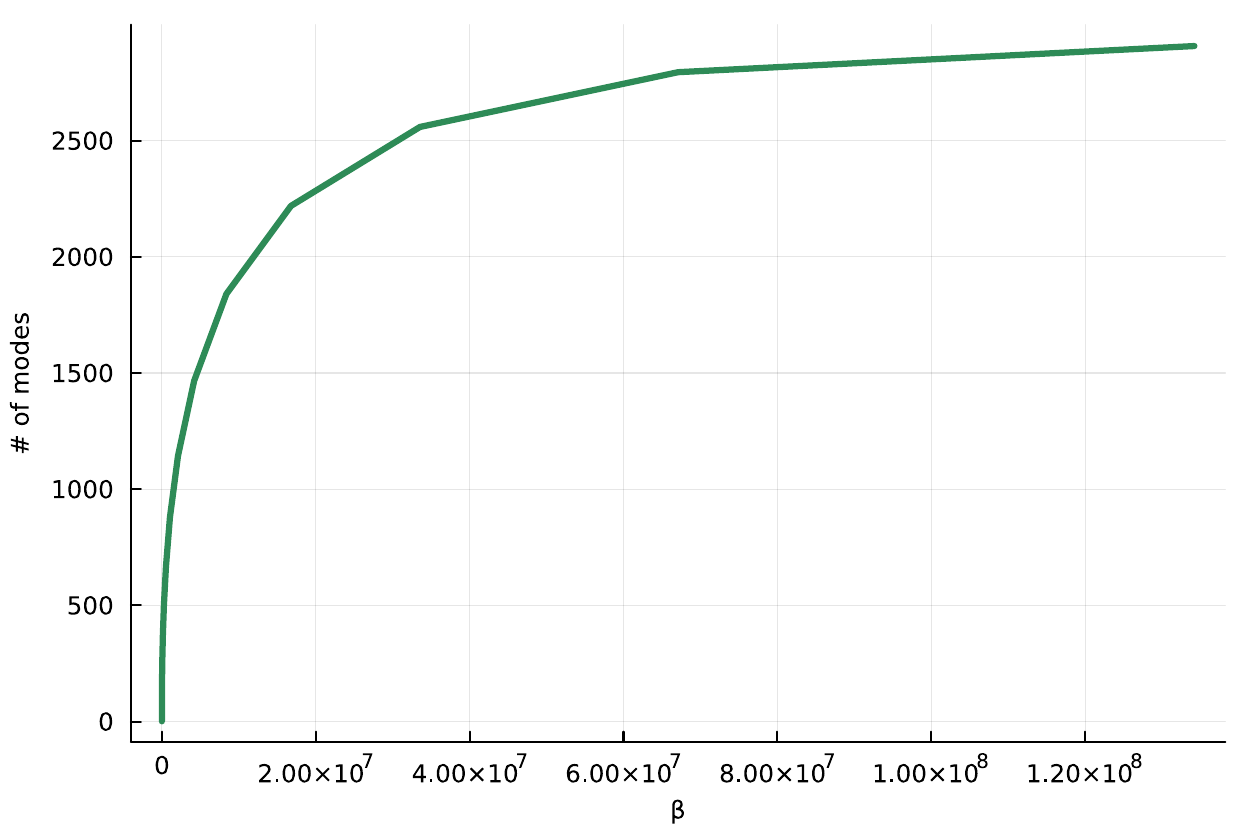}
    \hspace{0.1cm}
    \includegraphics[scale=0.285]{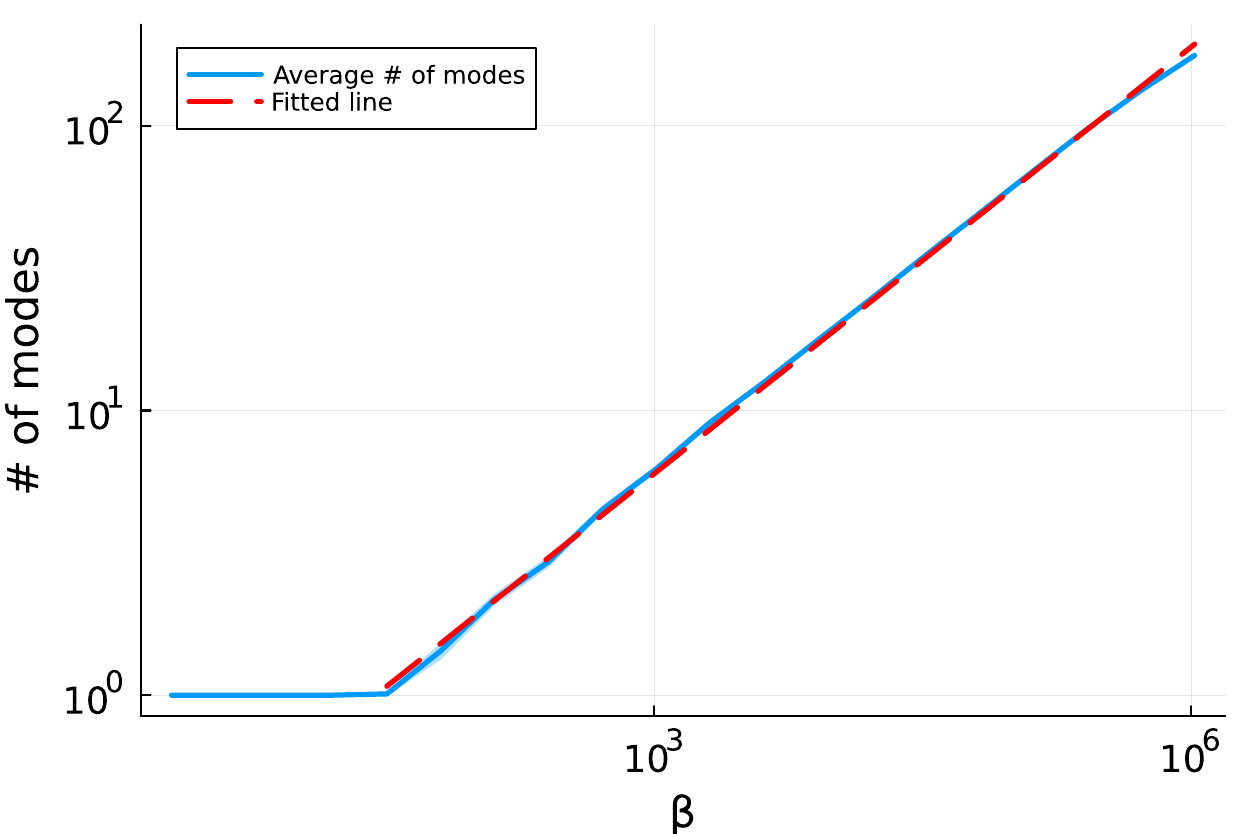}
    \caption{(Left) Plot of the average number of modes as a function of $\beta$ for $n=10^3$ (top) and $n=10^4$ (bottom). (Right) Log-log plot for $n=10^3$ (top) and $n=10^4$ (bottom); the predicted linear regression line ({\color{JuliaRed}red}) corroborates a power-law of the form $\text{average \# of modes} \approx 0.179 \cdot \beta^{0.504}$, in line with \Cref{thm:main-result}.}
    \label{fig: 1}
\end{figure}

\subsection{Motivation}
\label{sec:transformers}

The question of estimating the number of modes as a function of the bandwidth has a plethora of applications in statistical inference and multimodality tests---see \cite{mammen91, mammen95, mammen97} and the references therein. 
Another application which has stimulated some of the recent progress on the topic is data clustering. 
The latter can be achieved nonparametrically using a KDE, whose modes, and hence clusters, can be detected using the \emph{mean-shift algorithm} \cite{fukunaga1975estimation, cheng1995mean, comaniciu2002mean, carreira2000mode, carreira2003number, carreira2007gaussian, rodriguez2014clustering, carreira2015review}, which can essentially be seen as iterative local averaging. The main idea in mean-shift clustering is to perform a mean-shift iteration starting from each data point and then define each mode as a cluster, with all points converging to the same mode grouped into the same cluster. The analysis of this algorithm has led to upper bounds on the number of modes of \eqref{eq:gkde} \cite{carreira2003number}.

\begin{figure}[!ht]
    \centering
    \includegraphics[scale=0.425]{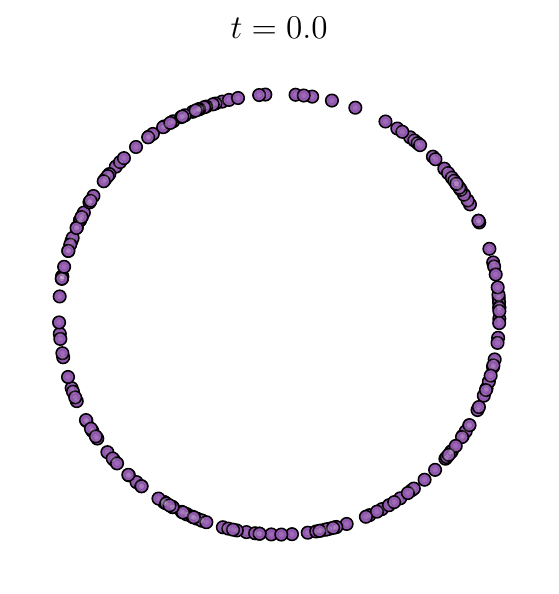}
    \includegraphics[scale=0.425]{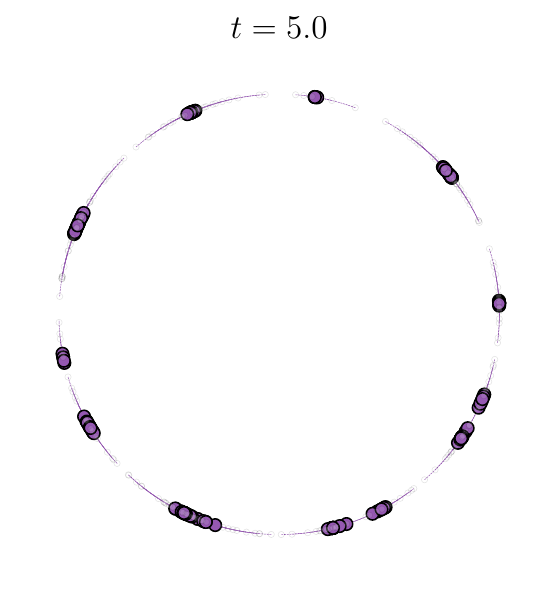}
    \includegraphics[scale=0.425]{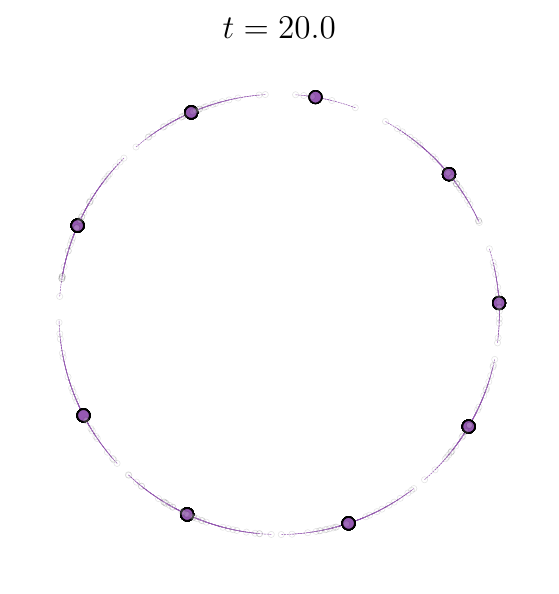}
    \caption{Metastability of self-attention dynamics at temperature $\beta=81$ initialized with $n$ iid uniform points on the circle, with $n=200$ (top) and $n=1000$ (bottom). The number of clusters appears of the correct order $\sim\sqrt{\beta}$. (Code available at \href{https://github.com/borjanG/2023-transformers-rotf}{\color{blue}github.com/borjanG/2023-transformers-rotf}.)}
    \includegraphics[scale=0.425]{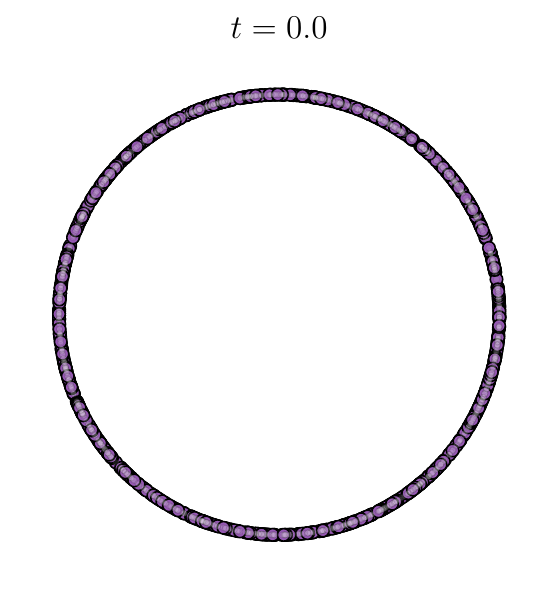}
    \includegraphics[scale=0.425]{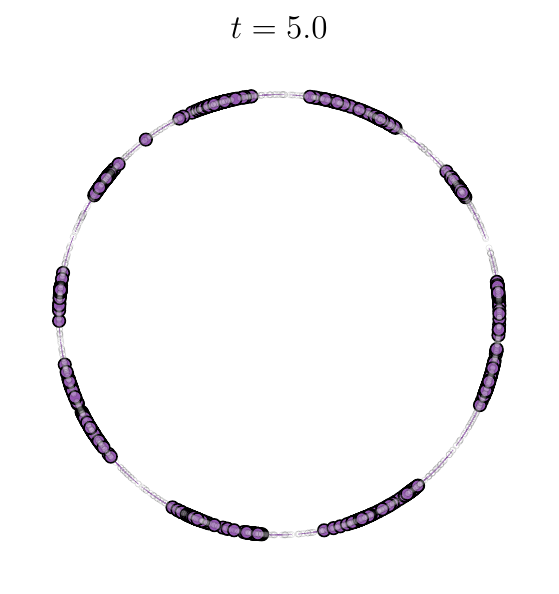}
    \includegraphics[scale=0.425]{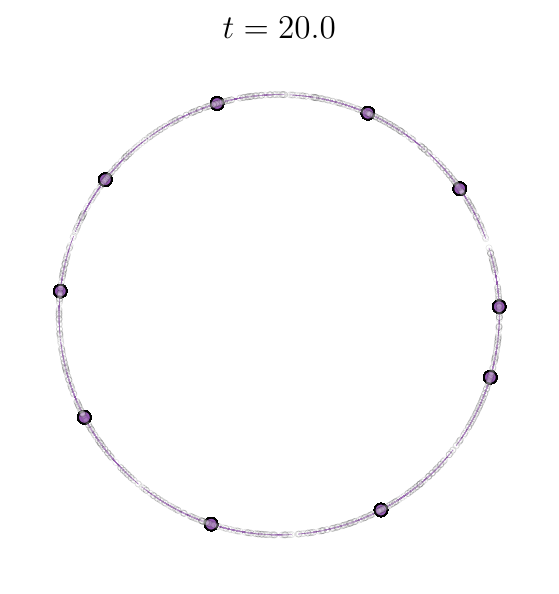}
    \label{fig: 2}
\end{figure}

We were instead brought to this problem from another perspective, motivated by the study of \emph{self-attention dynamics} \cite{sander2022sinkformers, geshkovski2023mathematical, geshkovski2024emergence, geshkovski2024measure}---a toy model for \emph{Transformers}, the deep neural network architecture that has driven the success of large language models \cite{vaswani2017attention}. These dynamics form a mean-field interacting particle system
\begin{equation*}
    \frac{\diff}{\diff\tau}x_i(\tau) = \sum_{j=1}^n \frac{e^{\beta\langle x_i(\tau), x_j(\tau)\rangle}}{\displaystyle \sum_{k=1}^n e^{\beta\langle x_i(\tau), x_j(\tau)\rangle}} \mathsf{P}_{x_i(\tau)}^\perp(x_j(\tau)),
\end{equation*}
evolving on the unit sphere $\mathbb{S}^{d-1}$ because of $\mathsf{P}_x^\perp \coloneqq I_d-xx^\top$. Here, $\tau\ge0$ plays the role of layers, the $n$ particles $x_i(\tau)$ represent tokens evolving through a dynamical system. This system is characterized by a temperature parameter $\beta \ge 0$ that governs the space localization of particle interations.
One sees that all particles move in time by following the field 
$\nabla \log (\mathsf{K}_{\beta^{-1/2}} \ast \mu_\tau)$; here, $\mu_\tau$ is the empirical measure of the particles $x_1(\tau), \ldots, x_n(\tau)$ at time $\tau$.

It is shown that for almost every initial configuration $x_1(0),\ldots,x_n(0)$, and for $\beta \ge 0$, 
all particles converge to a single cluster in infinite time~\cite{geshkovski2023mathematical, criscitiello2024synchronization, polyanskiy2025synchronization}. Rather than converging quickly, \cite{geshkovski2024dynamic} prove that the dynamics instead manifest \emph{metastability}: particles quickly approach a few clusters, remain in the vicinity of these clusters for a very long period, and eventually coalesce into a single cluster in infinite time. Concurrently, and using different methods,  \cite{bruno2024emergence} show a similar result: starting from a perturbation of the uniform distribution, beyond a certain time, the empirical measure of the $n$ particles approaches an empirical measure of $O(\sqrt{\beta})$-equidistributed points on the circle in the mean-field limit, and stays near it for long time. This is done by a study of the linearized system and leveraging nonlinear stability results from \cite{grenier2000nonlinear}. See also \cite{bruno2025multiscale, karagodin2024clustering, alcalde2025attention}.

\edit{
Our interest lies in the number of clusters during the first metastable phase in dimension $d=2$.
At time $\tau=0$, we initialize $n$ tokens at iid uniform points on the circle. Under the 
self-attention flow, tokens follow the vector field $\nabla\log\big(K_{\beta^{-1/2}} * \mu_\tau\big)$, 
so metastable clusters coincide with local maxima of the smoothed empirical measure 
$K_{\beta^{-1/2}} * \mu_\tau$. In particular, at early times the circle is partitioned by the 
stationary points of $K_{\beta^{-1/2}} * \mu_0$, and each arc contracts toward its nearest maximum, 
making the number of clusters equal to the number of these maxima. 
Our $1d$ analysis shows that, for iid Gaussian data, the maxima concentrate in two symmetric belts where $t^2 \in [2\log n - 3\log\beta,2\log n - \log\beta]$; on the circle this 
corresponds to angular locations where the local inter-token spacing is $\simeq \beta^{-\frac12}$, 
explaining both the $\sqrt{\beta}$ scaling of the metastable cluster count and their preferred 
positions.}

Here, we focus on a simplified setting by working on the real line instead of the circle (or higher-dimensional spheres), but we believe the analysis could be extended to these cases pending technical adaptations. Notwithstanding, \Cref{thm:main-result} reflects what is seen in simulations (\Cref{fig: 2}).

\subsection{Sketch of the proof} \label{sec: sketch}

The spirit of the proof of results such as \Cref{thm:mammen} and others presented in \cite{mammen91,mammen95,mammen97} is similar to ours---one applies the Kac-Rice formula (\Cref{thm:kac-rice}) to a Gaussian approximation of $(\wh{P}_n'(t), \wh{P}_n^{''}(t))$ and argues its validity. However, the main limitation of these works is that modes are counted in a fixed and finite interval $[a, b]$ (and $[0,1]^d$ in the higher dimensional cases). Extending these techniques to the whole real line demands for different, significantly stronger, approximation results using Edgeworth expansions.

We sketch the key ideas that allow us to count modes over the real line. 
\edit{We truncate $\mb{R}$ to the interval
\begin{equation}\label{eq:T}
T\coloneqq\left[-\sqrt{2\log n-\log\beta-\omega(\beta)}, \sqrt{2\log n-\log\beta-\omega(\beta)}\right]
\end{equation}
where $\omega$ is a fixed, slow growing function such that 
$$1\ll \omega(\beta)\ll \log\log \beta,$$ 
and so $T$ is well-defined for large $\beta$.
Motivated by \cref{thm:mammen,thm:main-result}, we also define the interval
\begin{equation}\label{eq:T'}
T'\coloneqq\left[-\sqrt{2\log n-3\log\beta}, \sqrt{2\log n-3\log\beta}\right]
\end{equation}
if $\beta \le n^{\frac23}$ and define $T'=\varnothing$ if $\beta >n^{\frac23}$.}
We use the Kac-Rice formula to compute the expected number of modes of $\wh{P}_n$ in the symmetric intervals $T$ and $T'$. All asymptotics are as $n, \beta\to \infty$.

\begin{proposition}\label{prop:main-int}
If $n^c\lesssim \beta \lesssim n^{2-c}$ for arbitrarily small $c>0$, then 
\begin{enumerate}
    \item In expectation over $X_i$, the number of modes of $\wh{P}_n$ in $T$ is
$\Theta\left( \sqrt{\beta\log \beta} \right)$. 
\item In expectation over $X_i$, the number of modes of $\wh{P}_n$ in $T'$ is
$O\left(\sqrt{\beta} \right)$. 
\end{enumerate}
\end{proposition}

The Kac-Rice computation appears tractable only when the joint distribution of $(\wh{P}_n'(t), \wh{P}_n''(t))$ is Gaussian, which it is not. To overcome this obstacle, we apply the Kac-Rice formula over a Gaussian approximation of the joint distribution in \cref{sec:normal}.
For the specific underlying density and KDE in \cref{eq:gkde}, we are able to justify in \cref{sec:error} the approximation for all $t$ in the growing interval $T$ 
instead of a fixed interval. This is why \cref{thm:mammen} only counts modes in a fixed interval.

To show the validity of the Gaussian approximation, we use the \emph{Edgeworth expansion} of the joint distribution of $(\wh{P}_n'(t), \wh{P}_n''(t))$ around the Gaussian distribution with matching first two moments. We bound the error due to the third order term of the expansion directly, and deal with the higher order terms by appealing to the error bounds of densities in the Edgeworth approximation similar to \cite[Theorem 19.2 and 19.3]{bhattacharya2010normal}. This strategy has been used in \cite{bally2019non}, but in a completely different context.
We note that \cite{mammen97} employ the same theorem to justify the Gaussian process approximation over a fixed interval.

\edit{
Indeed, as $T, \beta$ grow with $n$, the error decay rate of the Edgeworth approximation is quite delicate near the boundary of $T$. Instead of the usual case of powers of $n^{-\frac12}$, it is powers of $e^{-\frac{\omega(\beta)}{4}}$ (see \cref{lem:error-higher} and \cref{eq:rate}). This is exactly why we need to introduce the $\omega(\beta)$ term in $T$.}
In doing so, we will see that the normal approximation is invalid outside of $T$ (see \cref{rmk:delicate}), but crucially $T$ is sufficiently large to cover almost all modes, as observed empirically in \cref{rmk:empirical} and \Cref{figure:approx} and given below. 

\begin{proposition} \label{prop:main-tail}
If $n^c\lesssim \beta \lesssim n^{2-c}$ for arbitrarily small $c>0$, then the expectation over $X_i$ of the number of modes of $\wh{P}_n$ that lie outside of $T$ is
$O\left( e^{\frac{\omega(\beta)}{2}}\sqrt{\beta} \right)$.
\end{proposition}

We prove this in \cref{sec:tail} with an argument from scale-space theory: we bound the number of modes outside $T$ by the number of samples $X_i$ outside $T$, which we then bound naively. This is precisely the argument used by \cite[Theorem 2]{carreira2003number} to show Gaussian mixtures over $\mb{R}$ with $n$ components must have at most $n$ modes. This argument crucially relies on the kernel density estimator being Gaussian (see \cref{rmk:gkde}). 

Now, \cref{thm:main-result} follows from \cref{prop:main-int,prop:main-tail} provided $1\ll \omega(\beta)\ll \log \log \beta$.
\edit{Indeed, the bounds on $\omega(\beta)$ are chosen to balance these error terms. In fact, all error terms other than the Kac-Rice integral over $T\setminus T'$ of the Gaussian approximation of the density are $O\left( e^{-\frac{\omega(\beta)}{4}}\sqrt{\beta\log\beta} \right)$.}
\subsection{Notation}

We adopt standard notation from asymptotic analysis: we write $f(x)\ll g(x)$ or $f(x)=o(g(x))$ if $f(x)/g(x)\to 0$ as $x\to\infty$; $f(x) \lesssim g(x)$ or $f(x)=O(g(x))$ if there exists a finite, positive constant $C$ such that $f(x)\le Cg(x)$; and we write $f(x)\asymp g(x)$ or $f(x)=\Theta(g(x))$ if $f(x)\lesssim g(x)$ and $g(x)\lesssim f(x)$. 
\edit{We also write $f(x)\sim g(x)$ if $f(x)/g(x)\to 1$ as $x\to\infty$. Similarly, for vector and matrix-valued functions $\mbf{f}(x)\lesssim \mbf{g}(x)$ if $f_i(x)\lesssim g_i(x)$ for every entry, indexed by $i$. We use the analogous notation for $\mbf{f}(x)\asymp \mbf{g}(x)$ and $\mbf{f}(x)\sim \mbf{g}(x)$. Note that all asymptotic constants are absolute.}

\section{Kac-Rice for the normal approximation}
\label{sec:normal}

\subsection{The Kac-Rice formula}
\label{sec:kac-rice}

We say that $\Psi:\mb{R}\to\mb{R}$ has an \emph{upcrossing of level $u$} at $t\in\mb{R}$ if $\Psi(t)=u$ and $\Psi'(t)>0$. 
The Kac-Rice formula allows us to compute the expected number of up-crossings when $F$ is a random field (i.e., a stochastic process).

\begin{theorem}[Kac-Rice, {\cite[pp. 62]{azais2009level}, \cite[Section 11.1]{adler2009random}}]
\label{thm:kac-rice}
Consider a random $\Psi:\mb{R}\to\mb{R}$, some fixed $u\in\mathbb{R}$ and a compact $T\subset\mathbb{R}$. Suppose
\begin{enumerate}
    \item $\Psi$ is a.s. in $\mathscr{C}^1(\mathbb{R})$, and $\Psi, \Psi'$ both have finite variance over $T$;
    \item The law of $\Psi(t)$ admits a density $p_t^{[1]}(x)$ which is continuous for $t\in T$ and $x$ in a neighborhood of $u$;
    \item The joint law of $(\Psi(t), \Psi'(t))$ admits a density $p_t(x,y)$ which is continuous for $t\in T$, $x$ in a neighborhood of $u$, and $y\in\mb{R}$;
    \item $\mathbb{P}(\upomega(\eta)>\varepsilon)=O(\eta)$ as $\eta\searrow 0^+$ for any $\varepsilon>0$, where $\upomega(\cdot)$ denotes the modulus of continuity\footnote{defined, for $f:\mb{R}\to\mb{R}$, as $\upomega(\eta)= \sup_{t, s\colon |t-s|\le\eta}|f(t)-f(s)|$.} of $\Psi'(\cdot)$.
\end{enumerate} 
Define the number of up-crossings in $T$ of $\Psi$ at level $u\in\mathbb{R}$ as
\begin{equation*} 
U_u(\Psi, T) \coloneqq \left|\{t\in T:\Psi(t)=u, \Psi'(t)>0\}\right|.
\end{equation*}
Then, with expectation taken over the randomness of $\Psi$, 
\begin{equation}
\label{eq:krf}
\mb{E}U_u(\Psi, T) = \int_T \int_0^\infty yp_t(u, y)\diff y\diff t.
\end{equation}
\end{theorem}

The Kac-Rice formula extends to any dimension $d\ge 1$, and also on manifolds other than $\mathbb{R}^d$---see \cite[Section 11.1]{adler2009random}.
It is the classical tool for computing the expected number of critical points of random fields, with many recent applications including spin glasses \cite{auffinger2013random, fan2021tap} and landscapes of loss functions arising in machine learning \cite{maillard2020landscape}. While the method applies to general densities, the conditional expectation appears infeasible to compute or estimate beyond the Gaussian case. \edit{Moreover, we remark that our reliance on the Kac-Rice formula precludes us from any ``with high probability'' anlogs of \cref{thm:main-result}, though we do expect such statements to hold.}

For the KDE $\wh{P}_n$ defined in \cref{eq:gkde}, define the random function $F_n:\mb{R}\to\mb{R}$ by
\begin{equation}
\label{eq:Fn}
F_n(t) = \frac{1}{\sqrt{n}}\sum_{i=1}^n \left(t-X_i\right)e^{-\frac{\beta}{2}\left(t-X_i\right)^2}= -\sqrt{\frac{2\pi n}{\beta^3}}\wh{P}_n'(t).
\end{equation}
Then $t\in\mb{R}$ is an upcrossing of $F_n$ at level $0$ if and only if $F_n(t)=0$ and $F'_n(t)>0$. This is equivalent to $\wh{P}'_n(t)=0$ and $\wh{P}_n''(t)<0$, i.e. $t$ is a mode of $\wh{P}_n$. Thus, the number of modes of $\wh{P}_n$ in $T$ is given by $U_0(F_n, T)$. 
For $T, T'$ defined in \eqref{eq:T}--\eqref{eq:T'}, \cref{prop:main-int,prop:main-tail} \edit{are equivalent to}
\begin{equation}
\label{eq:main-eq-form}
\begin{aligned}
\mb{E}U_0\left(F_n, T\right) & \asymp \sqrt{\beta\log\beta},\\
\mb{E}U_0\left(F_n, T'\right) & \lesssim e^{-\frac{\omega(\beta)}{4}}\sqrt{\beta\log\beta},\\
\mb{E}U_0\left(F_n, \mb{R}\setminus T\right) & \lesssim e^{\frac{\omega(\beta)}{2}}\sqrt{\beta}.
\end{aligned}
\end{equation}

\subsection{Computing the Gaussian approximation}
Without loss of generality, fix $t\in T$ with $t\ge 0$.  
We can rewrite $F_n(t)$ from \cref{eq:Fn} and compute its derivative: for independent copies $(G_i, G_i')$ of
\begin{equation} \label{eq: Gt}
    \begin{bmatrix}
        G(t) \\ G'(t)
    \end{bmatrix}=
e^{-\frac{\beta}{2} (t-X)^2}\begin{bmatrix}
        t-X \\ 1-\beta (t-X)^2
     \end{bmatrix},
\end{equation}
where $X\sim N(0,1)$, we have
\begin{equation*} 
\begin{bmatrix}
        F_n(t) \\ 
        F_n'(t)
\end{bmatrix}
= \frac{1}{\sqrt{n}}\sum_{i=1}^n \begin{bmatrix}
        G_i(t) \\ G_i'(t)
    \end{bmatrix}\sim p_t.
\end{equation*}
We prove that $p_t$ is a well-defined density in \Cref{lem: pt.bdd}, and defer the following computation to \cref{sec:moments,sec:moments}.
\begin{lemma}
\label{lem:moments-p}
The mean and covariance matrix of the random vector $(F_n(t), F_n'(t))$ are given respectively by
\begin{equation}
\label{eq:moments-p}
\begin{aligned}
\mu_t &\coloneqq \sqrt{n}\begin{bmatrix}
\mb{E}G(t)\\ \mb{E}G'(t)
\end{bmatrix} \sim n^{\frac12}\beta^{-\frac32}e^{-\frac{t^2}{2}}\begin{bmatrix}
t \\ \edit{1}-t^2
\end{bmatrix} \\
\Sigma_t &\coloneqq \begin{bmatrix}
\on{Var}G(t) & \on{Cov}(G(t), G'(t))\\ \on{Cov}(G(t), G'(t)) & \on{Var}G'(t)
\end{bmatrix} \sim 2^{-\frac52}\beta^{-\frac32} e^{-\frac{t^2}{2}}\begin{bmatrix}
2 & -t \\ -t & 3\beta
\end{bmatrix}.
\end{aligned}
\end{equation}
\end{lemma}

We proceed to centering and rescaling the density $p_t$. Let $Y_i(t)$, $i\in[n]$, be independent copies of
\begin{equation}
\label{eq:Yi}
Y(t) = \Sigma_t^{-\frac12}\begin{bmatrix}
G(t) - \mb{E}G(t) \\ 
G'(t) - \mb{E}G'(t)
\end{bmatrix}
\end{equation}
Let $q_t$ denote the density of $n^{-\frac{1}{2}} \sum_{i=1}^n Y_i(t)$. By construction $q_t$ has mean $0$ and covariance $I_2$. Moreover, by the change-of-variables formula, it holds
\begin{equation}
\label{eq:qt}
p_t(x, y) = \left(\det\Sigma_t\right)^{-\frac12}q_t\left(\Sigma_{t}^{-\frac12}[(x, y)-\mu_t]\right).
\end{equation}
Now, let $\varphi:\mb{R}^2\to\mb{R}$ be the density of $N(0, I_2)$, i.e.,
\begin{equation*}
    \varphi(x) \coloneqq \frac{1}{{2\pi}} e^{-\frac{\Vert x\Vert^2}{2}}.
\end{equation*}
We aim to approximate the Kac-Rice integral \cref{eq:krf} as follows:
\begin{equation}
\label{eq:approx}
\int_T\int_0^\infty yp_t(0, y)\diff y \edit{\diff t}\approx \int_T\int_0^\infty y \left(\det\Sigma_t\right)^{-\frac12}\varphi\left(\Sigma_{t}^{-\frac12}[(0, y)-\mu_t] \right)\diff y\diff t.
\end{equation}
The validity of this approximation is deferred to \cref{sec:error}. In the remainder of this section, we solely focus on computing the right hand side integral. \edit{
\begin{lemma}
\label{lem:phi-t}
There exists $A_t\asymp \beta^{-\frac32}nt^2e^{-\frac{t^2}{2}}, \delta_t\asymp n^{\frac12}\beta^{-\frac32}e^{-\frac{t^2}{2}}(1-t^2/2)$, and $\alpha_t\asymp \beta^{\frac{1}{2}}e^{\frac{t^2}{2}}$ such that 
\begin{equation}
\label{eq:phi-t}
\begin{aligned} 
\left\Vert \Sigma_{t}^{-\frac12}[(0, y)-\mu_t] \right\Vert^2 &\sim A_t+\alpha_t(y-\delta_t)^2,\\
\int_0^\infty y \varphi\left(\Sigma_{t}^{-\frac12}[(0, y)-\mu_t] \right)\diff y & \asymp \alpha_t^{-1}e^{-A_t}.
\end{aligned}
\end{equation}
\end{lemma}
\begin{proof}[Proof of \Cref{lem:phi-t}]
We recall \cref{eq:moments-p} to compute
$$
\Omega\coloneqq\Sigma_t^{-1}\sim 3^{-1}2^{\frac32}\beta^{\frac12}e^{\frac{t^2}{2}} \begin{bmatrix}
3\beta & t \\ t &2
\end{bmatrix}
$$ 
We let the prefactor to be $\alpha_t/2$. Since we kept leading coefficients of entries of $\Omega$ and $\mu_t$ up to a global absolute constant that we absorb in $\alpha_t$ and $\delta_t$, it is safe to verify leading coefficients do not cancel and compute asymptotically:
\begin{align*}
\left\Vert \Sigma_{t}^{-\frac12}[(0, y)-\mu_t]\right\Vert^2 & = \left\langle (-\mu_{t, 1}, y-\mu_{t, 2}), \Sigma_t^{-1}(-\mu_{t, 1}, y-\mu_{t, 2})\right\rangle
\\ & = \Omega_{11}\mu_{t, 1}^2-2\Omega_{12}\mu_{t, 1}(y-\mu_{t, 2})+\Omega_{22}(y-\mu_{t, 2})^2
\\ & \sim \alpha_t\mu_{t, 1}^2\left[\frac{3\beta}{2}-t\left(\frac{y}{\mu_{t, 1}}+\frac{t^2-1}{t}\right)+\left(\frac{y}{\mu_{t, 1}}+\frac{t^2-1}{t}\right)^2\right]
\\ & \sim \alpha_t\mu_{t, 1}^2\left(\frac{3\beta}{2}-\frac{t^2}{4}\right)+\alpha_t \mu_{t, 1}^2 \left(\frac{y}{\mu_{t, 1}}+\frac{t^2-1}{t}-\frac{t}{2}\right)^2
\\ & \sim \frac{3}{2}\beta\alpha_t\mu_{t, 1}^2+ \alpha_t \left(y-\frac{\mu_{t, 1}}{t}\left(1-\frac{t^2}{2}\right)\right)^2
\end{align*}
Now, we let $A_t$ be the first term and let $\delta_t$ be the term subtracting $y$. Verifying the asymptotics of both, we obtain the first statement in \cref{eq:phi-t}.
For the second statement, we have
\[
    \int_0^\infty y \varphi\left(\Sigma_{t}^{-\frac12}[(0, y)-\mu_t] \right)\diff y \sim e^{-A_t}\int_0^\infty y e^{-2\alpha_t(y-\delta_t)^2} \diff y\asymp\alpha_t^{-1}e^{-A_t}
\]
by a standard fact (\cref{lem:gaussian-int}) in Gaussian integrals, upon checking $\alpha_t^{-1}\delta_t^2\ll 1$ in our parameter regime of $t\in T$ and $n\lesssim \beta^{2-c}$. The statement and proof of the fact  is in \cref{sec:proof-gaussian-int}.
\end{proof}}

\subsection{\texorpdfstring{The Kac-Rice integral over $\varphi$}{The Kac-Rice}} \label{sec: 2.3}

\begin{figure}[ht!]
    \centering
    \includegraphics[scale=0.29]{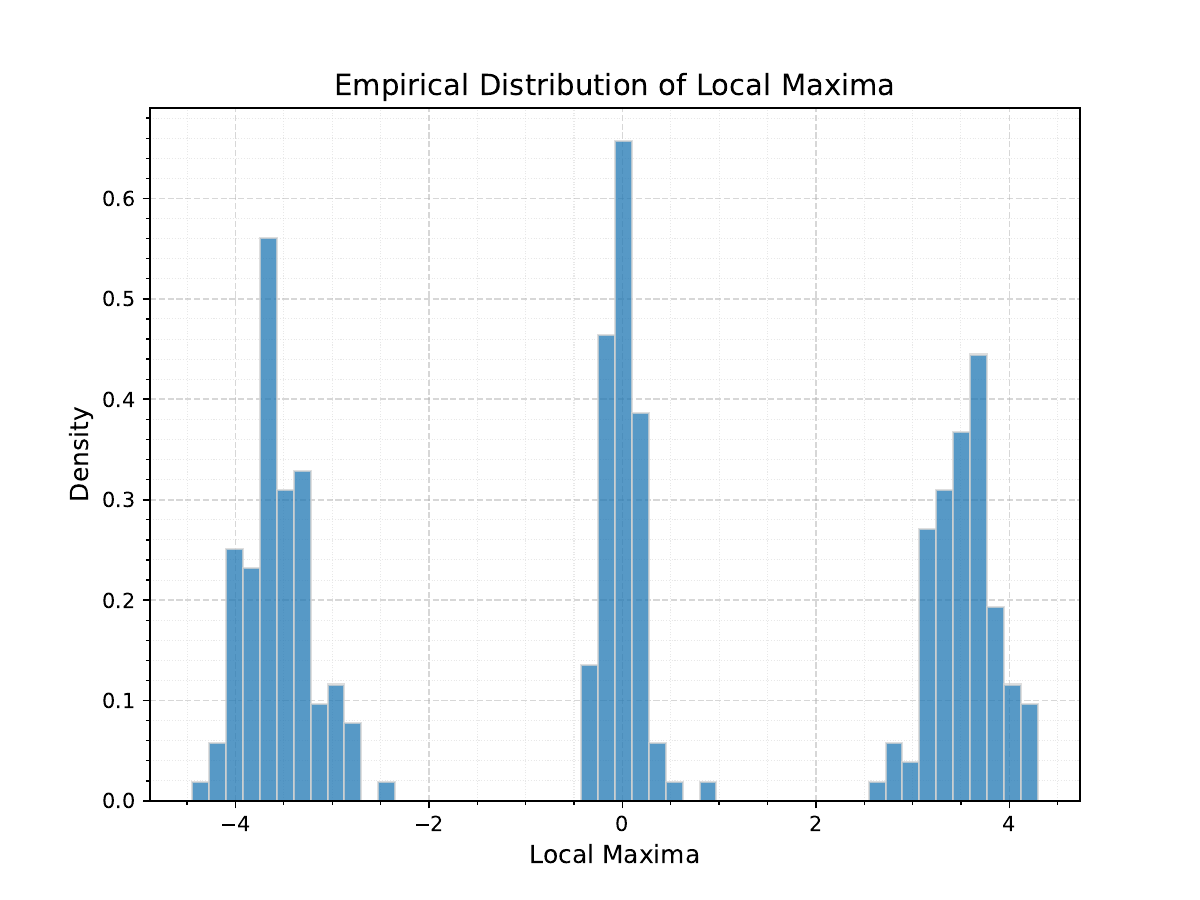}
    \hspace{0.2cm}
    \includegraphics[scale=0.29]{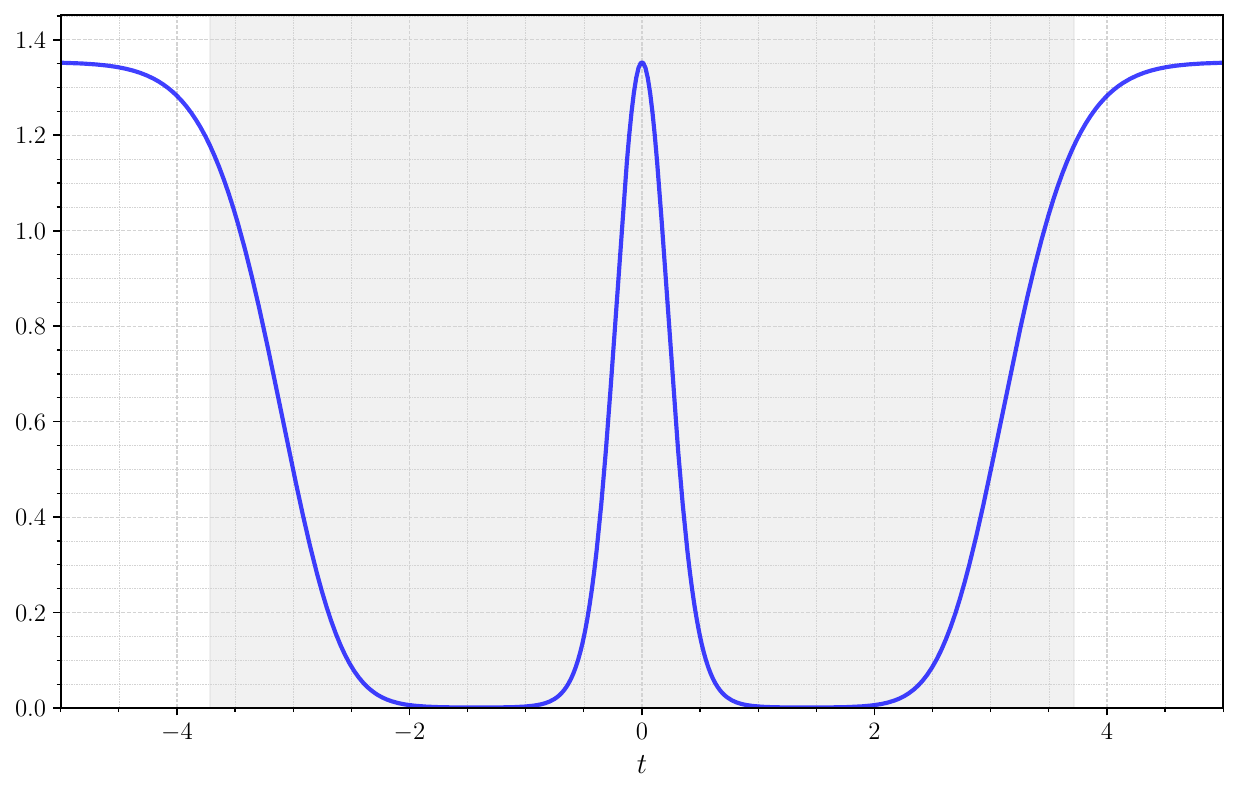}
    \includegraphics[scale=0.29]{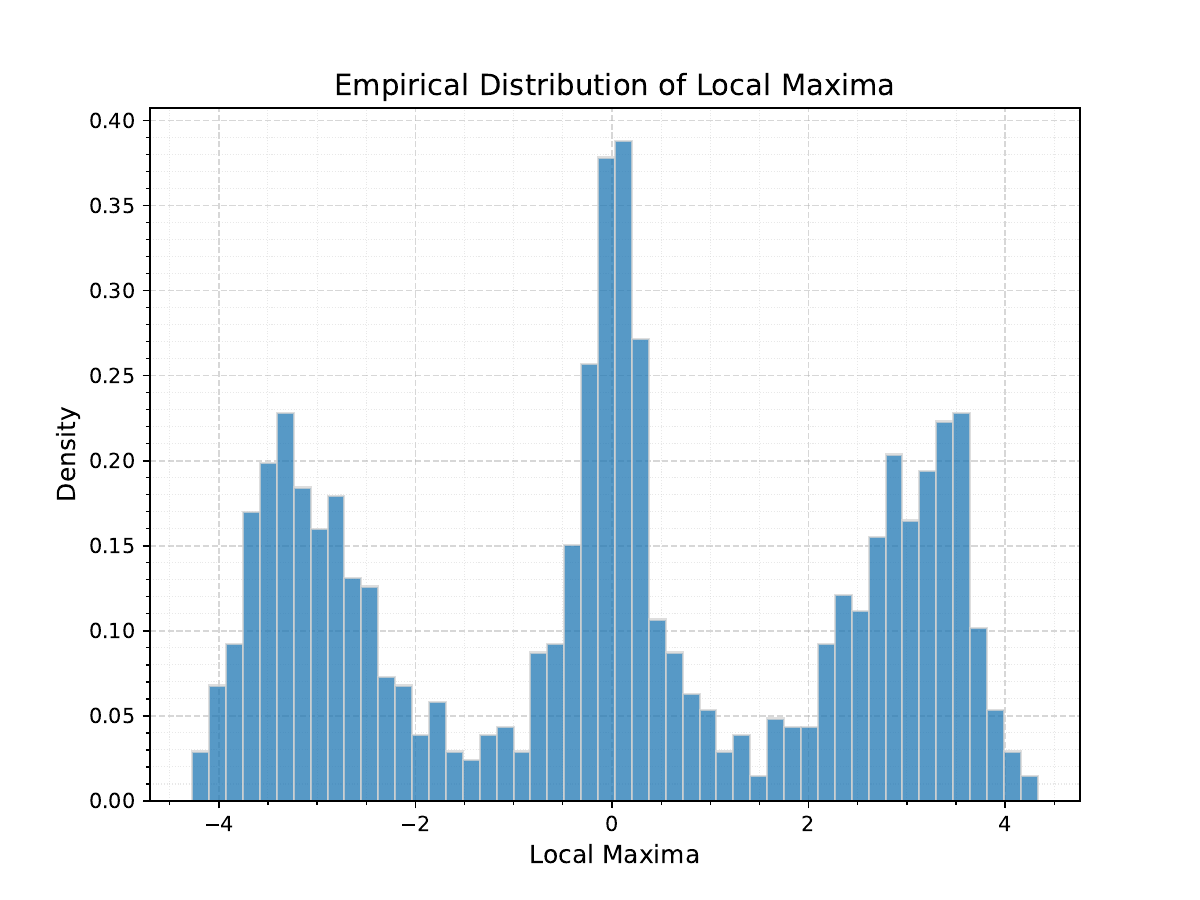}
    \hspace{0.2cm}
    \includegraphics[scale=0.29]{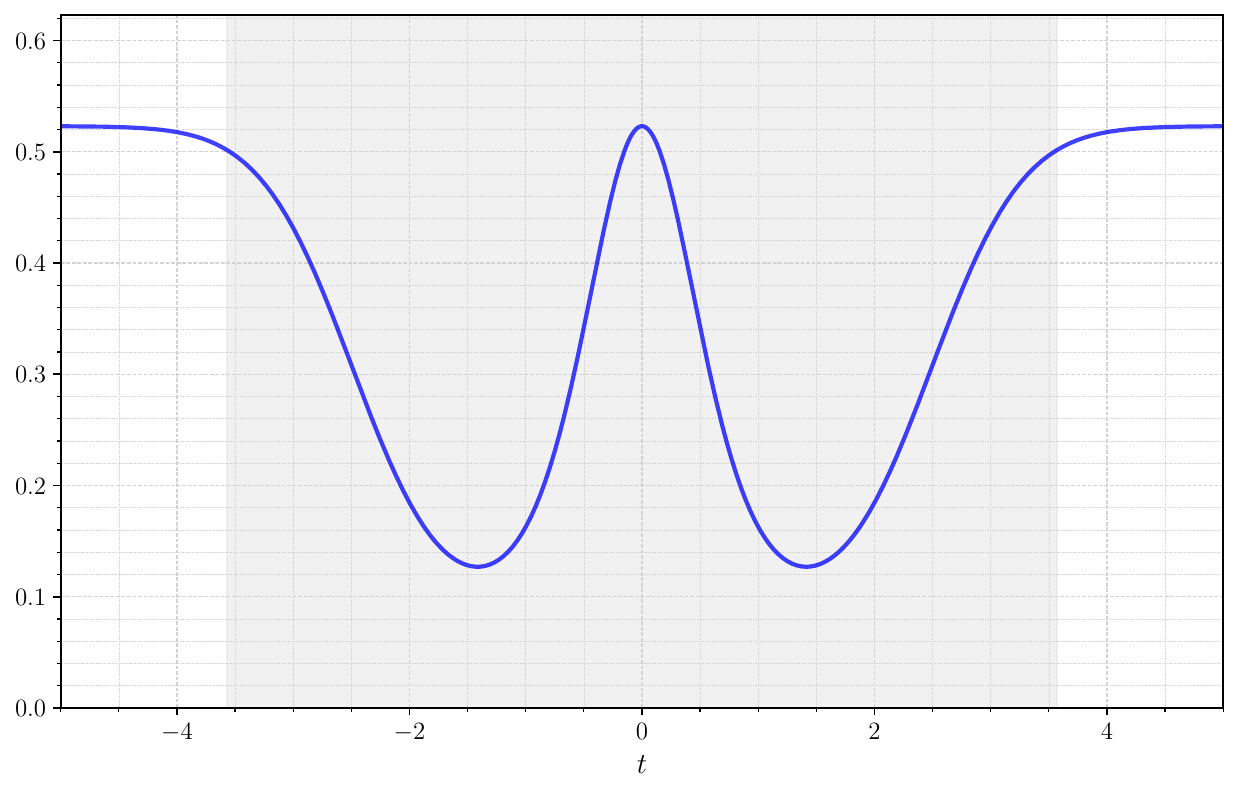}
    \caption{$n=10^5$ is fixed throughout.
    (Left) Empirical distribution of the modes of $\widehat P_n$ over $T$ for $\beta=100$ (top) and $\beta=300$ (bottom). (Right)
    The function $t\mapsto \sqrt{\beta}\exp(-A_t)$ for $\beta=100$ (top) and $\beta=300$ (bottom), which, due to the Kac-Rice formula, is an approximation for the distribution of the number of modes of $\widehat P_n$ in $T$. Shaded in grey is the interval $T$. (Code available at \href{https://github.com/KimiSun18/2024-gauss-kde-attention}{\color{blue}github.com/KimiSun18/2024-gauss-kde-attention}.)}
    \label{figure:approx}
\end{figure}

We compute \cref{eq:krf} under the approximation \cref{eq:approx}. By \cref{eq:phi-t} and \cref{eq:moments-p}, we have that
\begin{equation} 
\label{eq:int-phi-final}
\int_S\int_0^\infty y \left(\det\Sigma_t\right)^{-\frac12}\varphi\left(\Sigma_{t}^{-\frac12}[(0, y)-\mu_t] \right)\diff y\diff t\asymp \sqrt{\beta}\int_S e^{-A_t} \diff t
 \end{equation}
for any measurable $S\subset \mb{R}$. 
Assuming validity of the Gaussian approximation (see \cref{sec:error}), it follows from the Kac-Rice formula that the density of modes at $t \in \mb{R}$ is proportional to $\sqrt{\beta}e^{-A_t}$. We plot this density in \Cref{figure:approx} with the same choice of $n$ and $\beta$ as in the empirical distribution. 
We see that they match on the highlighted interval $T$, but not outside of $T$ where the Gaussian approximation \edit{is no longer valid}---see \cref{rmk:delicate}. 

We compute \cref{eq:int-phi-final} explicitly for $S=T$ and $S=T'$.

\begin{lemma}
\label{lem:main-int-phi}
If $n^c\lesssim \beta \lesssim n^{2-c}$ for some $c>0$, then
\begin{align*}
\int_T\int_0^\infty y \left(\det\Sigma_t\right)^{-\frac12}\varphi\left(\Sigma_{t}^{-\frac12}[(0, y)-\mu_t] \right)\diff y\diff t & \asymp  \sqrt{\beta\log\beta}, \\
\int_{T'}\int_0^\infty y \left(\det\Sigma_t\right)^{-\frac12}\varphi\left(\Sigma_{t}^{-\frac12}[(0, y)-\mu_t] \right)\diff y\diff t & \lesssim  \sqrt{\beta}.
\end{align*}
\end{lemma}

\begin{proof}[Proof of \Cref{lem:main-int-phi}]
Recall $A_t$ from \cref{lem:phi-t}. By \cref{eq:int-phi-final}, it suffices to show that
\begin{equation}
\label{eq:int-phi-b}
\int_Te^{-A_t}\diff t \asymp \sqrt{\log \beta}\quad\text{and}\quad 
\int_{T'} e^{-A_t}\diff t \lesssim 1.
\end{equation}
As $A_t>0$, the integral is at most the length of $T$, which is $O(\sqrt{\log\beta})$ by \cref{eq:T}. 
\edit{
Recall that $n\lesssim\beta^{2-c}$. For this constant $c>0$ and $k=1, 2$, define 
\[ t_k \coloneqq \sqrt{2\log n - \left(1+\frac{c}{10^k}\right)\log \beta}\]
Then, $t_k>0$ and $t_k\in T$ for both $k$. Now, as the integrand is positive, for constants $C, C'>0$
\begin{align*}
\int_Te^{-A_t}\diff t & \ge \int_{t_1}^{t_{2}}\exp\left(-C\beta^{-\frac32}nt^2e^{-\frac{t^2}{2}}\right)\diff t
\\ & \ge \left(t_2-t_1\right) \exp\left(-C\beta^{-\frac32}nt_{2}^2e^{-\frac{t_1^2}{2}}\right)
\\ & \gtrsim \sqrt{\log \beta}\exp \left(-C'\beta^{-\frac12+\frac{c}{10}}\log n\right)
\\ & \gtrsim \sqrt{\log\beta}
\end{align*}
as $n, \beta\to\infty$ with $\log n\asymp \log\beta$, and $c\in (0, 2]$. 
Now if $t\in T'$, we have 
$$
e^{-\frac{t^2}{2}}\ge \exp\left(\log n- \frac{3}{2}\log\beta\right)=\beta^{\frac32}n^{-1}.
$$ 
Hence
\[
\int_{T'} e^{-A_t}\diff t  \lesssim \int_{T'}\exp\left(-C\beta^{-\frac32}nt^2e^{-\frac{t^2}{2}}\right)\diff t
 \le \int_{-\infty}^{\infty}e^{-Ct^2}\diff t
\lesssim 1.\qedhere\]}
\end{proof}


\section{Leveraging the Edgeworth expansion}
\label{sec:error}
In this section, we show the approximation of $\edit{q_t}$ by $\varphi$ is valid in $T$ by showing
\begin{equation}
\label{eq:error-goal}
\int_T\int_0^\infty \left(\det\Sigma_t\right)^{-\frac12}y\left|q_t-\varphi\right|\left(\Sigma_{t}^{-\frac12}[(0, y)-\mu_t]\right)\diff y\diff t \ll \sqrt{\beta\log \beta}.
\end{equation}
One natural idea is to use some asymptotic series to expand $\edit{q_t}$ around $\varphi$, e.g. the Edgeworth expansion, to argue that $|q_t-\varphi|\ll \varphi$ in the sense of the integral over $y$. \edit{To this end, we cite a standard result in normal approximation theory in the case of identity covariance matrix, that we will follow closely,
\begin{theorem}[{\cite[Theorem 19.2 and 19.3]
{bhattacharya2010normal}}]
\label{thm:br}
Let $X_n$ be a sequence of i.i.d. random vectors in $\mb{R}^k$ with mean zero and identity covariance matrix. Suppose $\mb{E}\Vert X_1\Vert^{s+1}<\infty$ for some integer $s\ge 2$, then under suitable conditions, the density $q_n$ of $n^{-1/2}(X_1+\dots +X_n)$ admits the asymptotic expansion
\[ 
\sup_{\mbf{x}\in \mb{R}^k} \left(1+\Vert \mbf{x}\Vert^{s}\right)\left|q_n(\mbf{x}) -\sum_{j=0}^{s-2} n^{-\frac{j}{2}} Q_j(\mbf{x})\right| \lesssim n^{-\frac{s-1}{2}}
\]
as $n\to\infty$, where $Q_j$ is the $j$-th term of the Edgeworth expansion. In particular, $Q_0=\varphi$.
\end{theorem}

It is tempting to directly apply this theorem with $s=2$ to control $| q_t-\varphi|$ in \cref{eq:error-goal}.
}
We discuss the two major obstacles we have to overcome in order to implement this approach.
\begin{itemize}
    \item First, \cref{thm:br} and similar results on validity of asymptotic series such as Edgeworth expansions treat densities $q_t$ and $\varphi$ that are independent of $n$. As $\beta$ grows in $n$, we will need to re-derive these results and carefully track the dependence on $\beta$. This will give extra constraints on $(t, \beta, n)$ for the validity of such an asymptotic series. Fortunately, this will be satisfied precisely when $t\in T$. \edit{The analog of \cref{thm:br} in our setting for $s=2, 3$ is given as \cref{lem:error-higher}.}
    
    \edit{\item Second, if we directly apply \cref{thm:br} with $s=2$ to control $| q_t-\varphi|$, then the inner integral over $y\in [0, \infty)$ in \cref{eq:error-goal} fails to converge as it is of the form 
\begin{equation}
\label{eq:tilde-y-fail}
\int_0^\infty \frac{\tilde{y}}{1+\tilde{y}^2}\diff\tilde{y}\quad\text\quad \text{where}\quad \tilde{y}\coloneqq\alpha_t^{\frac12}(y-\delta_t)
\end{equation}}
\end{itemize}

\edit{
To overcome this obstacle, we 
use the above definition of $\tilde{y}$ and \cref{lem:phi-t} to obtain that
\[ \left\Vert \Sigma_{t}^{-\frac12}[(0, y)-\mu_t]\right\Vert^2 \sim A_t+\tilde{y}^2\]
This naturally suggests casework on which term has the dominant contribution: for $|\tilde{y}|\le \sqrt{A_t}$, we follow \cref{thm:br} for the $s=3$ case to bound the error of $q_t-\varphi$, whereby \cref{eq:tilde-y-fail} integrated from $y=0$ up to $\tilde{y}=\sqrt{A_t}$ converges; for $\tilde{y}\ge \sqrt{A_t}$, we go to the next term $n^{-1/2}\psi$ in the Edgeworth series, and manually control this third order error in \cref{sec:error-3}. Finally, we control the higher order terms following \cref{thm:br} for the $s=3$ case: there, the analog of \cref{eq:tilde-y-fail} for controlling
\[ \int_0^\infty y\left|q_t-\varphi-n^{-\frac12}\psi\right|\left(\Sigma_{t}^{-\frac12}[(0, y)-\mu_t]\right)\diff y\]
converges as $s=3$. This is done in \cref{sec:error-higher}.}

\subsection{Bounding the third order error for small $y$}
\label{sec:error-3}
Recall $Y$ from \cref{eq:Yi}. \edit{Let $\varphi$ denote the density of $N(0,I_2)$ and
$$H^\alpha(x)\coloneqq(-1)^{|\alpha|}\,\varphi(x)^{-1}\partial^\alpha \varphi(x)$$ the standard multivariate
Hermite polynomials for a multi-index $\alpha\in\mb{Z}_{\ge 0}^2$. 
Writing $q_t$ for the density of $n^{-\frac12}\sum_{i=1}^n Y_i(t)$, the third‑order
Edgeworth expansion is the multivariate Hermite expansion of $q_t/\varphi$:
\[
\frac{q_t(\mbf{x})}{\varphi(\mbf{x})}
= 1+\frac{1}{\sqrt{n}} \sum_{|\alpha|=3} \frac{\kappa_t^\alpha}{\alpha!}H^\alpha(\mbf{x})
+ r_{n,t}(\mbf{x}),
\]
so the next term is $n^{-1/2}\psi$ with
\begin{equation}\label{eq:psi}
\psi(\mbf{x})=\varphi(\mbf{x})\sum_{k=0}^3 \frac{\kappa_t^{(k,3-k)}}{k!(3-k)!}H^{(k,\,3-k)}(\mbf{x}).
\end{equation}
Here $\kappa_t^\alpha$ denotes the (order-$|\alpha|$) cumulant of the single‑sample vector
$Y(t)$, i.e. the $\alpha$‑th mixed derivative at $0$ of the cumulant generating function
$\log\mathbb{E}e^{\langle u,Y(t)\rangle}$. For $|\alpha|=3$ and our normalization, one has the equivalent  moment identity
\[
\kappa_t^\alpha= \mathbb{E}\left[H^\alpha(Y(t))\right]
=\mathbb{E}_{Z\sim N(0,I_2)}\left[\frac{q_t(Z)}{\varphi(Z)}H^\alpha(Z)\right].
\]}
\edit{If $|\alpha|\coloneqq\alpha_1+\alpha_2=s$, then $\kappa_t^\alpha$ can be bounded above asymptotically by the $s$-th moments of $\Vert Y\Vert$, which we bound in \cref{sec:moments}.}
\begin{lemma}
\label{lem:eta}
For $s\ge 3$, the cumulants of $Y$ with order $|\alpha|=s$ satisfy
\[\kappa_t^\alpha\lesssim \eta_s\quad\text{where}\quad \eta_s \coloneqq \mb{E}[\Vert Y\Vert^s] \lesssim \left(\beta e^{t^2}\right)^{\frac{s-2}{4}}.\]
\end{lemma}

\edit{
Trivially, $|H^{(k, 3-k)}(\mbf{x})|\lesssim \Vert\mbf{x}\Vert^3$.
By \cref{lem:phi-t}, we know for $y\ge \Delta_t\coloneqq{\delta_t+\sqrt{A_t/\alpha_t}}$ that 
\[ \tilde{y}\coloneqq \alpha_t^{\frac{1}{2}}(y-\delta_t)\asymp \left\Vert\Sigma_{t}^{-\frac12}[(0, y)-\mu_t]\right\Vert \]
Therefore, for any $t\in T$ and $y\ge\Delta_t$, we can bound
\begin{equation}
\label{eq:H3-bound}
\left|H^{(k, 3-k)}\left(\Sigma_{t}^{-\frac12}[(0, y)-\mu_t]\right)\right|\lesssim \Vert\tilde{y}\Vert^3.
\end{equation}
}

Now, by a similar method as \cref{lem:phi-t}, we obtain the following bound. It says that when we integrate the Edgeworth series $\edit{q_t} = \varphi +n^{-\frac12}\psi+\dots$ over \edit{$y\ge\Delta_t$ and $t\in T$}, the contribution $\varphi$ dominates $n^{-\frac12}\psi$, hinting at the validity of the approximation.
\begin{lemma}
\label{lem:error-3}
Recall $T, T'$ from \cref{eq:T,eq:T'}, and $A_t$ from \cref{lem:phi-t}. \edit{Let  $\Delta_t\coloneqq\delta_t+\sqrt{A_t/\alpha_t}$. Then
\begin{align*}
\int_T\int_{\Delta_t}^\infty y\left(n\det\Sigma_t\right)^{-\frac12}\psi\left(\Sigma_{t}^{-\frac12}[(0, y)-\mu_t]\right)\diff y\diff t &\lesssim e^{-\frac{\omega(\beta)}{4}} \sqrt{\beta\log\beta},\\
\int_{T'}\int_{\Delta_t}^\infty y\left(n\det\Sigma_t\right)^{-\frac12}\psi\left(\Sigma_{t}^{-\frac12}[(0, y)-\mu_t]\right)\diff y\diff t &\lesssim\sqrt{\beta}.
\end{align*}}
\end{lemma}
\edit{\begin{proof}[Proof of \Cref{lem:error-3}]
Note that $y\ge \Delta_t$ if and only if $\tilde{y}\coloneqq \alpha_t^{\frac{1}{2}}(y-\delta_t)\ge\sqrt{A_t}$. By \cref{lem:phi-t,lem:eta}, we can combine bounds \cref{eq:psi,eq:H3-bound} to obtain
\begin{align*}
& \int_T\int_{\Delta_t}^\infty y\left(n\det\Sigma_t\right)^{-\frac12}\psi\left(\Sigma_{t}^{-\frac12}[(0, y)-\mu_t]\right)\diff y\diff t 
\\ &\lesssim \sum_{k=0}^3 \int_T \left(n\det\Sigma_t\right)^{-\frac12} \kappa_t ^{(k, 3-k)}\int_{\Delta_t}^\infty y\left[\varphi H^{(k, 3-k)}\right]\left(\Sigma_{t}^{-\frac12}[(0, y)-\mu_t]\right) \diff y\diff t
\\ & \lesssim \int_T \left(n\det\Sigma_t\right)^{-\frac12} \eta_3 e^{-A_t} \int_{\Delta_t}^\infty y e^{-\frac{\tilde{y}^2}{2}}|\tilde{y}|^3\diff y\diff t
\\ & \asymp \int_T \left(n\det\Sigma_t\right)^{-\frac12} e^{-A_t}\eta_3 \alpha_t^{-1}\left(\int_{\sqrt{A_t}}^\infty \tilde{y}^{4}e^{-\frac{\tilde{y}^2}{2}}d\tilde{y}\right)\diff t
\\ & \lesssim n^{-\frac12}\beta^{\frac34}\sup_{t\in T}\left(e^{\frac{t^2}{4}}\right)\int_T e^{-A_t}\diff t
\\ & \lesssim  e^{-\frac{\omega(\beta)}{4}} \sqrt{\beta\log\beta}
\end{align*}
where we apply \cref{eq:int-phi-b}. The second statement for $T'$ holds similarly by considering $\sup_{t\in T'}\left(e^{\frac{t^2}{4}}\right)$.
\end{proof}}
\begin{remark}
\label{rmk:delicate}
One can see at this is actually an asymptotic equality by checking the Gaussian integrals in the proof above are of their typical order (i.e. no cancellation of leading terms). Hence, the decay is only a factor of $e^{-\omega (\beta)/4}$.
For $t\not\in T$, even $t=\sqrt{2\log n-0.99\log\beta}$, the last inequality in \cref{lem:error-3} fails and we get a bound of polynomially larger that $\sqrt{\beta}$. Then, as the third order error is asymptotically larger than the contribution of the Gaussian approximation, so the normal approximation is no longer valid.
This can be seen by comparing the plots in \Cref{figure:approx}. 
\end{remark}

\subsection{Bounding higher order errors}
\label{sec:error-higher}
We follow the classical proof of \cref{thm:br} about the validity of the Edgeworth expansion as an asymptotic series to show bound the higher order pointwise error of density function as follows.
\edit{\begin{lemma}
\label{lem:error-higher}
Suppose $n^c\lesssim \beta \lesssim n^{2-c}$ for some $c>0$. Let $g_2 \coloneqq q_t-\varphi$ and $g_3\coloneqq q_t-\varphi-n^{-\frac12}\psi$. Then
\begin{equation}
\label{eq:error-higher}
\sup_{\mbf{x}\in \mb{R}^2}\left(1+\Vert\mbf{x}\Vert^s\right)\left|q_t-\varphi\right|(\mbf{x}) \lesssim n^{-\frac{s-1}{2}}\eta_{s+1}
\end{equation}
as $n, \beta\to\infty$, for both $s=2$ and $s=3$, where we recall $\eta_s$ from \cref{lem:eta}.
\end{lemma}}
We 
defer the proof to \cref{sec:pf-error-higher}. \edit{Here, we comment on the differences with \cref{thm:br}: there, it is shown that order $s$ error is at most $n^{-\frac{s-1}{2}}$ provided $\eta_{s+1}$ is bounded. In our case, we pick up an extra $\eta_{s+1}$ factor as it is dependent on $\beta$. By \cref{lem:eta}, this bound is at most
\begin{equation}
\label{eq:rate}
n^{-\frac{s-1}{2}}\eta_{s+1} \lesssim \left(n^{-1}\beta^{\frac12}e^{\frac{t^2}{2}}\right)^{\frac{s-1}{2}}\lesssim \begin{cases}
    e^{-\frac{(s-1)\omega(\beta)}{4}} &\textup{if }t\in T\\
    \beta^{-(s-1)/2} &\textup{if }t\in T'
\end{cases}
\end{equation}
by definition of $T$ and $T'$.
This is the key reason for the extra $\omega(\beta)$ term in $T$: we get a small by non-negligible decay rate that matches \cref{lem:error-3}. Using \cref{lem:error-higher}, we control the Kac-Rice integrals.}
\begin{corollary}
\label{cor:error-higher}
\edit{Let $\Delta_t\coloneqq\delta_t+\sqrt{A_t/\alpha_t}$.} If $n^c\lesssim\beta\lesssim n^{2-c}$ for $c>0$, then asymptotically in $n, \beta\to\infty$
\edit{\begin{align*}
\int_T\int_0^{\Delta_t} \left(\det\Sigma_t\right)^{-\frac12}y\left|q_t-\varphi\right|\left(\Sigma_{t}^{-\frac12}[(0, y)-\mu_t]\right)\diff y\diff t  &\lesssim e^{-\frac{\omega(\beta)}{4}} \sqrt{\beta\log\beta}\\
\int_{T}\int_0^\infty \left(\det\Sigma_t\right)^{-\frac12}y\left|q_t-\varphi-n^{-\frac12}\psi\right|\left(\Sigma_{t}^{-\frac12}[(0, y)-\mu_t]\right)\diff y\diff t  &\lesssim e^{-\frac{\omega(\beta)}{2}} \sqrt{\beta\log\beta}.
\end{align*}
Moreover, the left hand sides of both displays above with $T$ replaced by $T'$ are $O(\sqrt{\beta})$.
}
\end{corollary}
\begin{proof}[Proof of \Cref{cor:error-higher}]
\edit{By \cref{lem:error-higher,lem:phi-t}, we let $\tilde{y}\coloneqq \alpha_t^{\frac{1}{2}}(y-\delta_t)$ as before to obtain
\begin{align*}
& \int_T\int_0^{\Delta_t} \left(\det\Sigma_t\right)^{-\frac12}y\left|q_t-\varphi\right|\left(\Sigma_{t}^{-\frac12}[(0, y)-\mu_t]\right)\diff y\diff t
\\ & \lesssim e^{-\frac{\omega(\beta)}{4}}\int_T\int_0^{\Delta_t}  \left(\det\Sigma_t\right)^{-\frac12} y\left(1+\left\Vert \Sigma_{t}^{-\frac12}[(0, y)-\mu_t]\right\Vert ^2\right)^{-1}\diff y\diff t
\\ & \lesssim e^{-\frac{\omega(\beta)}{4}}\int_T\int_0^{\Delta_t}  \left(\det\Sigma_t\right)^{-\frac12} \left(\frac{y}{1+A_t+\tilde{y}^2}\right)\diff y\diff t
\\ & \asymp e^{-\frac{\omega(\beta)}{4}}\int_T \left(\det\Sigma_t\right)^{-\frac12}\alpha_t^{-1} \left(\int_{-\sqrt{A_t}}^{\sqrt{A_t}} \frac{|\tilde{y}|}{1+A_t +\tilde{y}^2}\diff\tilde{y}\right)\diff t
\\ & \asymp e^{-\frac{\omega(\beta)}{4}}\int_T \left(\det\Sigma_t\right)^{-\frac12}\alpha_t^{-1} \log\left(\frac{1+2A_t}{1+A_t}\right)\diff t
\\ & \lesssim e^{-\frac{\omega(\beta)}{4}}\sqrt{\beta}\int_T\log(2)dt
\\ & \lesssim e^{-\frac{\omega(\beta)}{4}}\sqrt{\beta\log\beta}
\end{align*}
where we check that if $y=0$, then $\tilde{y}=-\alpha_t^{\frac12}\delta_t$ and $\alpha_t^{\frac12}\delta_t\ll A_t^{\frac12}$, so we may symmetrize the integral over $\tilde{y}$ up to a constant factor. Now similarly for the second display
\begin{align*}
& \int_T\int_0^\infty \left(\det\Sigma_t\right)^{-\frac12}y\left|q_t-\varphi-n^{-\frac12}\psi\right|\left(\Sigma_{t}^{-\frac12}[(0, y)-\mu_t]\right)\diff y\diff t
\\ & \lesssim e^{-\frac{\omega(\beta)}{2}}\int_T\int_0^\infty \left(\det\Sigma_t\right)^{-\frac12} y\left(1+\left\Vert \Sigma_{t}^{-\frac12}[(0, y)-\mu_t]\right\Vert ^3\right)^{-1}\diff y\diff t
\\ & \lesssim e^{-\frac{\omega(\beta)}{2}}\int_T\int_0^\infty \left(\det\Sigma_t\right)^{-\frac12} y\left(1+\left(A_t+\tilde{y}^2\right)^{\frac32}\right)^{-1}\diff y\diff t
\\ &\lesssim e^{-\frac{\omega(\beta)}{2}}\int_T \left(\det\Sigma_t\right)^{-\frac12} \alpha_t^{-1}\left(\int_{-\infty}^\infty \frac{|\tilde{y}|}{1+|\tilde{y}|^3}\diff\tilde{y}\right)\diff t
\\ & \asymp e^{-\frac{\omega(\beta)}{2}}\sqrt{\beta}\int_T \frac{2\pi}{3\sqrt{3}}dt
\\ & \asymp e^{-\frac{\omega(\beta)}{2}}\sqrt{\beta\log\beta}
\end{align*}

Now, the exact same computation but with $T$ replaced by $T'$ gives bounds $O(\sqrt{\beta})$ upon replacing exponential in $\omega(\beta)$ decay rates with those in \cref{eq:rate} for \cref{lem:error-higher}.
}
\end{proof}
\edit{In particular, by non-negativity of the integrand, the second equation holds upon replacing the bounds of integration of $y$ from $y\ge 0$ to $y\ge\Delta_t$. This is the form we will use.
}

\section{\texorpdfstring{Proof of \Cref{thm:main-result}}{Proof of}}

We prove \cref{prop:main-int,prop:main-tail} by checking \cref{eq:main-eq-form}, thereby proving \cref{thm:main-result}.

\subsection{\texorpdfstring{Proof of \cref{prop:main-int}}{Proof of}}
\label{sec: kac.rice.application}

To prove \Cref{prop:main-int} we seek to apply \Cref{thm:kac-rice}  to $\mb{E}U_0(F_n, T)$. This in turn requires checking all the assumptions of \Cref{thm:kac-rice}. 
We have

\edit{
\begin{proposition} \label{lem: pt.bdd}
Fix any \(\beta>0\), an integer \(n\ge 5\), and $t\in T$. 
Let $\upmu_t$ denote the law of $(F_n(t), F_n'(t))$ defined in \eqref{eq:Fn}. Then $\upmu_t$ admits a density $p_t\in\mathscr{C}^0(\mb{R}^2)$ satisfying $p_t({\bf{x}})\to0$ as $\|\bf{x}\|\to\infty$. 
Moreover, conditions {\it 1, 2, 4} in \Cref{thm:kac-rice} also hold for $\Psi(t) = F_n(t)$, thus \Cref{thm:kac-rice} applies.
\end{proposition}
}

We defer the proof to \Cref{sec: proof.pt.bdd}.
With \Cref{lem: pt.bdd}, 
we deduce

\begin{lemma} \label{lem:kr-appl}
With the notation as in \cref{thm:kac-rice},
\begin{equation}
\mb{E}U_0\left(F_n, T\right) = \int_T\int_0^\infty yp_t(0, y)\diff y\diff t.
\end{equation}
\end{lemma}

\begin{proof}[Proof of \cref{prop:main-int}]
We decompose the Kac-Rice integral as follows:
\edit{\begin{align*}
\mb{E}U_0\left(F_n, T\right) & = \int_T\int_0^\infty yp_t(0, y)\diff y\diff t
\\ & = \int_T\int_0^\infty \left(\det\Sigma_t\right)^{-\frac12}y q_t\left(\Sigma_{t}^{-\frac12}[(0, y)-\mu_t]\right)\diff y\diff t
\\ & = \int_T\int_0^\infty \left(\det\Sigma_t\right)^{-\frac12}y\varphi\left(\Sigma_{t}^{-\frac12}[(0, y)-\mu_t]\right)\diff y\diff t
\\ & \quad + \int_T\int_0^{\Delta_t} \left(\det\Sigma_t\right)^{-\frac12}y\left[q_t-\varphi\right]\left(\Sigma_{t}^{-\frac12}[(0, y)-\mu_t]\right)\diff y\diff t
\\ & \quad + \int_T\int_{\Delta_t}^\infty \left(n\det\Sigma_t\right)^{-\frac12}y\psi\left(\Sigma_{t}^{-\frac12}[(0, y)-\mu_t]\right)\diff y\diff t
\\ & \quad + \int_T\int_{\Delta_t}^\infty \left(\det\Sigma_t\right)^{-\frac12}y\left[q_t-\varphi-n^{-\frac12}\psi\right]\left(\Sigma_{t}^{-\frac12}[(0, y)-\mu_t]\right)\diff y\diff t
\\ & \asymp \sqrt{\beta\log \beta}
\end{align*}

Now, by \cref{lem:main-int-phi}, the first summand is $\Theta\left(\sqrt{\beta\log \beta}\right)$, while the last three are $O\left(e^{-\frac{\omega(\beta)}{4}}\sqrt{\beta\log \beta}\right)$ by \cref{lem:error-3,cor:error-higher}. 
Replacing $T$ by $T'$, all four summands are $O(\sqrt{\beta})$, as desired.}
\end{proof}

\begin{figure}[!ht]
    \centering
    \includegraphics[scale=0.3]{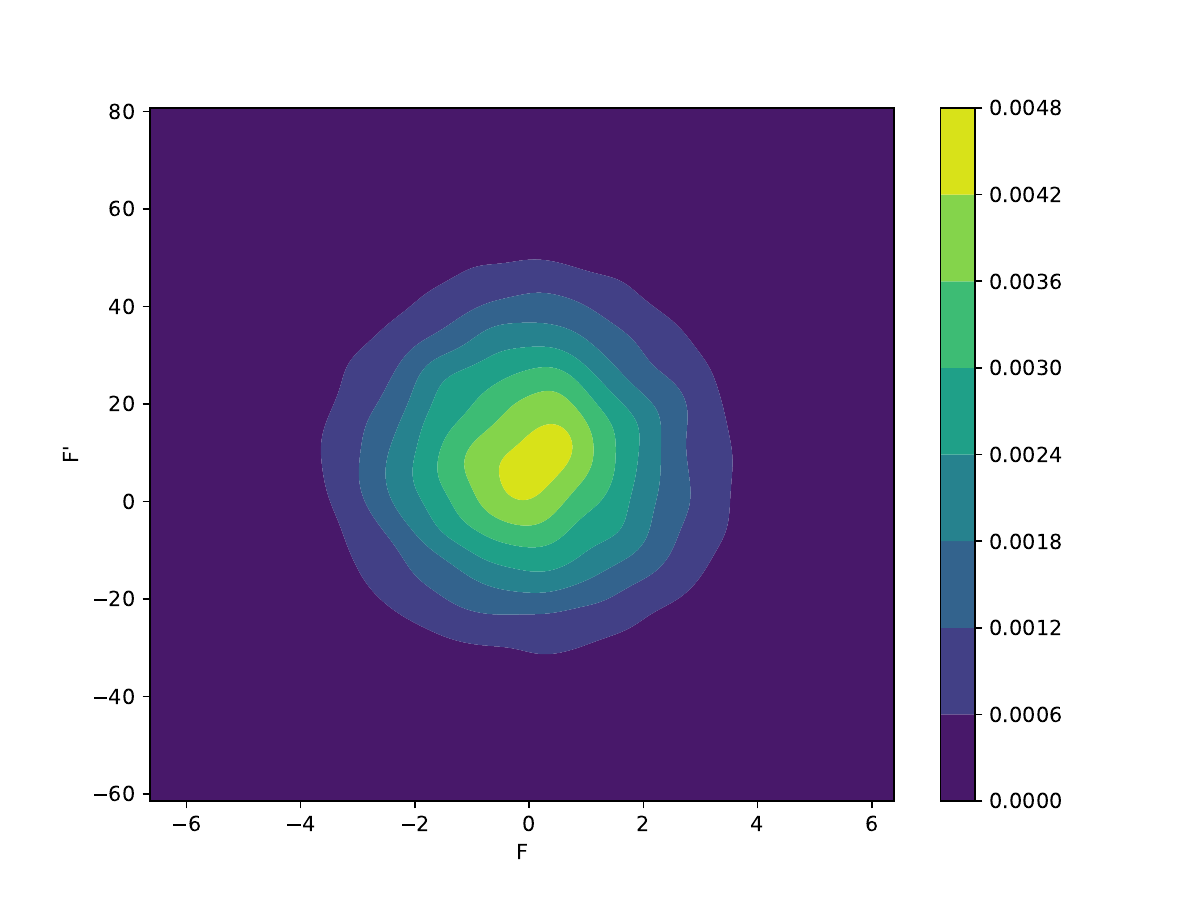}
    \includegraphics[scale=0.3]{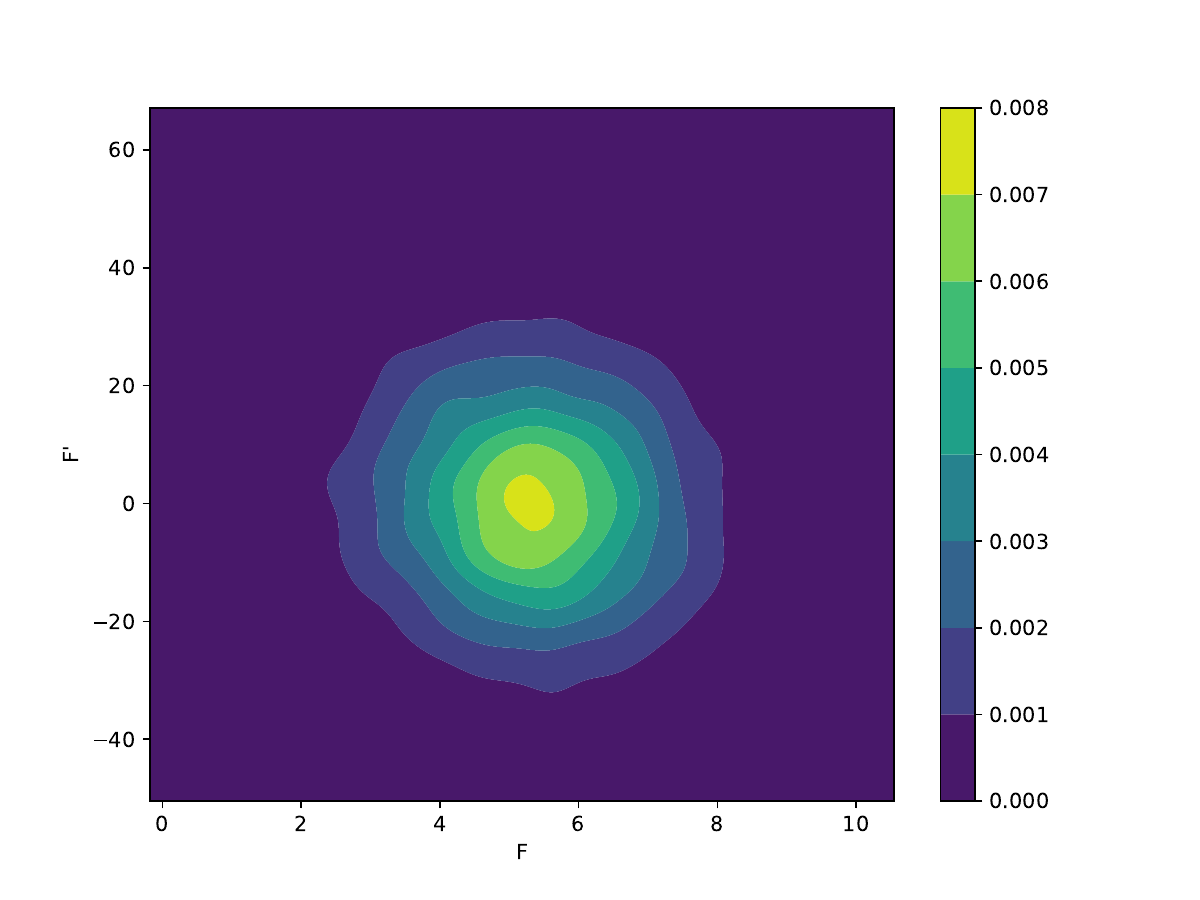}
    \includegraphics[scale=0.3]{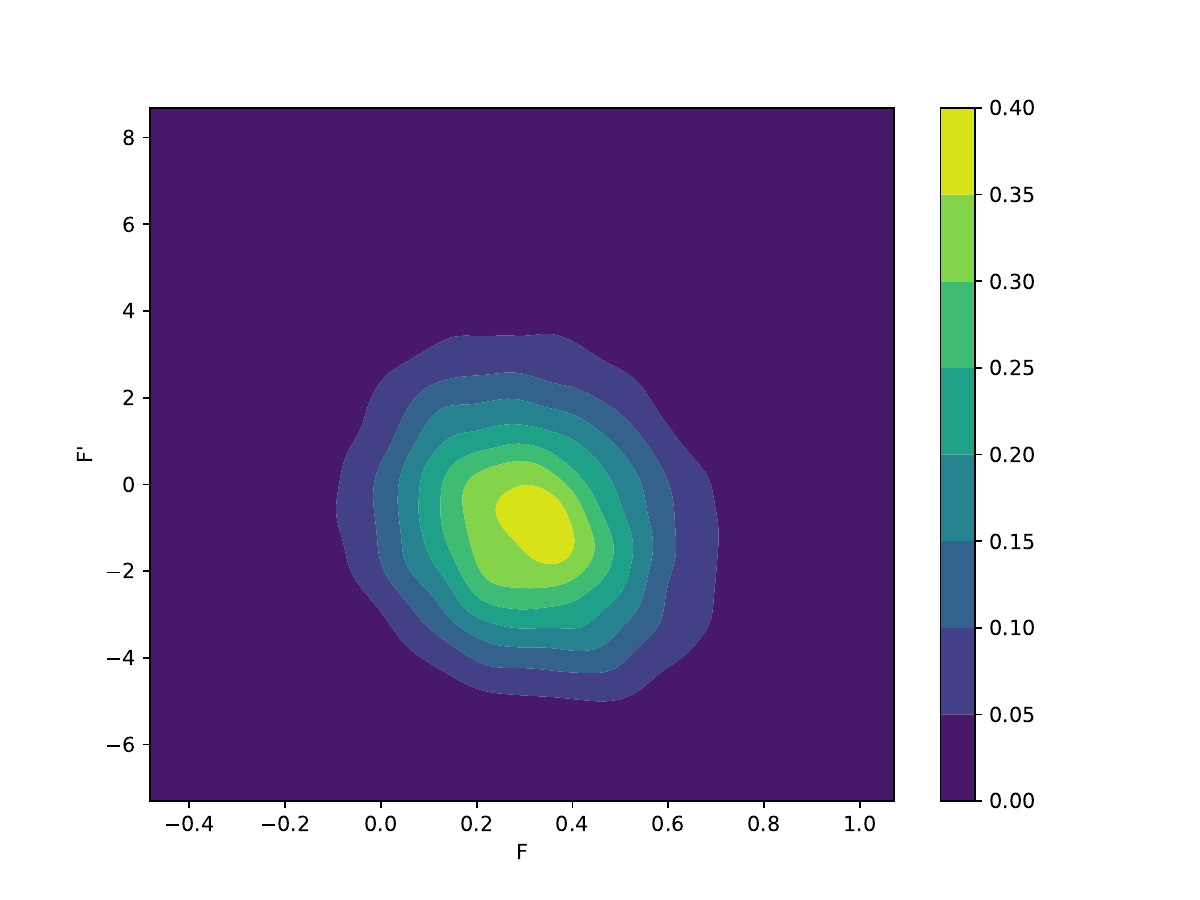}
    \includegraphics[scale=0.3]{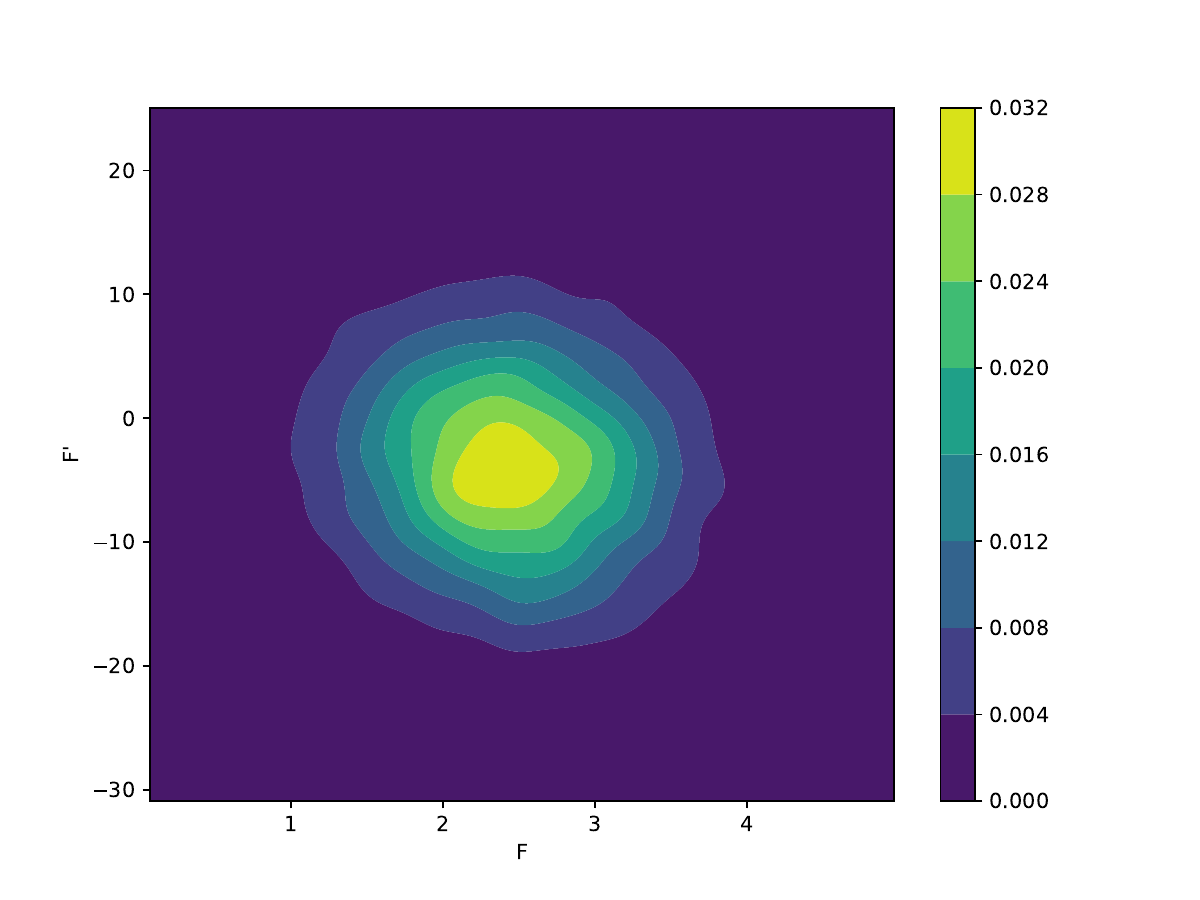}
    \caption{An estimate of the density $p_t=p_t(x,y)$ of $(F_n(t), F_n'(t))$ for $t=0, 1, 2, 3$ (clockwise from top left), where $\beta=81$ and $n=6500$, so that $\sqrt{2\log n -\log \beta}\approx 3$. (Code available at \href{https://github.com/KimiSun18/2024-gauss-kde-attention}{\color{blue}github.com/KimiSun18/2024-gauss-kde-attention}.)}
    \label{fig: kimi}
\end{figure}

\subsection{\texorpdfstring{Proof of \cref{prop:main-tail}}{Proof of}}
\label{sec:tail}

In this section, we prove \cref{prop:main-tail}.
\begin{lemma}\label{lem:scale-space}
For any $a>0$ and $X_1, \dots , X_n\in \mb{R}$, the number of modes of $\wh{P}_n$ in $(a, \infty)$ is at most $|I|$ where $I=\{i\in [n]:X_i\ge a\}$. \edit{By symmetry, the same estimate holds for modes in $(-\infty, -a)$.}
\end{lemma}

\begin{proof}[Proof of \Cref{lem:scale-space}]
Note that $\wh{P}_n (t) =\sum_{i=1}^n g_i(t)$ where for $i\in [n]$ we define
\begin{equation} g_i(t)\coloneqq\sqrt{\frac{\beta}{2\pi n^2}} \mathsf{K}_{\beta^{-1/2}}\left(t-X_i\right).\end{equation}
For $i\not\in I$, $g_i$ is monotonically decreasing on $[X_i, \infty)\supset (a, \infty)$, so
$\sum_{i\not\in I}g_i(t)$ has no modes in $(a, \infty)$. To this Gaussian mixture, we add in $g_i(t)$ for $i\in I$ one-by-one. By \cite[Theorem 2]{carreira2003number}, each time the number of modes in $(a, \infty)$ increases by at most one. In $|I|$-many steps, there are at most $|I|$ such modes.
\end{proof}
\begin{remark}
Two remarks are in order:
\label{rmk:gkde}
\begin{itemize}
    \item \edit{As discussed in \cite{carreira2003number}, the scale-space property of the Gaussians allow us to view adding a component to the Gaussian mixture as adding a delta distribution to the mixture, which adds one mode, and applying a Gaussian blurring that does not create new modes.}
    \item This argument crucially relies on the KDE being Gaussian: as discussed in \cite{carreira2003number}, the Gaussian kernel is the only kernel where for any fixed samples the number of modes of the KDE is non-increasing in the bandwidth $h$, \edit{which enables the blurring step}. For other kernels, we do suspect the analog of \cref{lem:scale-space} to hold, but a different argument is needed. In particular, \cite{mammen91,mammen95,mammen97} avoids this problem by counting modes on compact sets.
\end{itemize}
\end{remark}
\begin{proof}[Proof of \cref{prop:main-tail}]
By \cref{lem:scale-space}, symmetry of $T$ in \cref{eq:T} around $t=0$, linearity of expectations, and the tail bound $\mb{P}(|X|\ge a)\le 2e^{-a^2/2}$ for $X\sim N(0,1)$
\begin{equation} 
\begin{aligned}
    \mb{E}U_0\left(F_n, \mb{R}\setminus T\right) & \le \mb{E}\left|\{i:X_i\not\in T\} \right| 
    \\ & = n\mb{P}\left(X\not\in T\right) 
    \\ &\le 2n\exp\left(-\frac{2\log n-\log\beta-\omega(\beta)}{2}\right)
    \\ & = 2\sqrt{\beta \exp(\omega(\beta))}
    \\ & \ll \sqrt{\beta\log\beta}
\end{aligned}
\end{equation}
by the definition of $\omega(\beta)$, proving \cref{prop:main-tail}.
\end{proof}
Having proven \cref{prop:main-int,prop:main-tail}, we conclude \cref{thm:main-result}.

\section{Additional proofs}
\subsection{\texorpdfstring{Proof of \Cref{lem:gaussian-int}}{Proof of}}
We frequently make use of the following standard exercise on Gaussian integrals and dominated convergence. 

\begin{lemma}
\label{lem:gaussian-int}
Let $\Gamma$ denote the Gamma function.
For any $\alpha >0$ and integer $n\ge 0$,
\[\int _{0}^{\infty }\edit{u}^{n}e^{-\alpha u^{2}}\diff u=\frac{1}{2}\Gamma\left(\frac{n+1}{2}\right) \alpha^{-\frac{n+1}{2}}.\]

\edit{Moreover, for a fixed $n$, as $\epsilon\to 0$
\[
 \int_0^\infty v^n e^{-(v-\epsilon)^2}dv \to \int_0^\infty v^n e^{-v^2}dv\quad\text{and}\quad
 \int_{-\infty}^\infty |v|^n e^{-(v-\epsilon)^2}dv \to 2\int_0^\infty v^n e^{-v^2}dv
\]}

\end{lemma}

\begin{proof}
\edit{
For the first display, by a change of variables $v=\alpha u^2$, we get
\[\int _{0}^{\infty }{u}^{n}e^{-\alpha u^{2}}\diff u=\frac{1}{2}\alpha^{-\frac{n+1}{2}} \int_0^\infty v^{\frac{n-1}{2}}e^{-v}\diff v\]
and we recognize the integral as the definition of the Gamma function. 

For the second display, we apply the dominated convergence theorem: clearly, the integrands converges pointwise as $\varepsilon\to 0$, and it is dominated by an integrable function that is $e2^n$ for $v\in [-2, 2]$ and $v^ne^{-v^2/4}$ for $v\ge 2$ as we may take $\varepsilon \le 1 \le |v|/2$.
}\end{proof}

\label{sec:proof-gaussian-int}
\subsection{\texorpdfstring{Proof of \Cref{lem:moments-p}}{Proof of}}

\label{sec:moments}

In this section we compute the first two moments of $(G, G')$ to prove \cref{lem:moments-p}. 
Note that if $n^c\lesssim \beta\lesssim n^{2-c}$ for some $c>0$, and for $t\in T$, then we have $\exp \Theta(t^2/\beta) \to 1$. This implies that exponentials in the moments are asymptotically $e^{-{t^2}/{2}}$.

We first compute $\mu_t$. Completing the square gives
\[ \frac{\beta}{2}z^2+\frac{1}{2}(z-t)^2 = \frac{\beta+1}{2}u^2+\frac{\beta t^2}{2(\beta+1)}\quad \text{where}\quad u = z-\frac{t}{\beta+1}.\] 
Hence, using \cref{lem:gaussian-int} we compute
{ \[\begin{aligned}\mb{E}G(t)  = \int_{-\infty}^\infty ze^{-\frac{\beta}{2}z^2}\diff\edit{\varphi(z-t)}& = \frac{1}{\sqrt{2\pi}}\int_{-\infty}^\infty  ze^{-\frac{\beta}{2} z^2 -\frac12 (z-t)^2}\diff z\\
& = \frac{e^{-\frac{\beta}{2(\beta+1)}t^2}}{\sqrt{2\pi}}\int_{-\infty}^\infty \left(u+\frac{t}{\beta+1}\right)e^{-\frac{\beta+1}{2}u^2}\diff u \\ 
& = \frac{e^{-\frac{\beta}{2(\beta+1)}t^2}}{\sqrt{2\pi}} \left(\frac{t}{\beta+1}\right) \frac{\sqrt{\pi}}{\left(\frac{\beta+1}{2}\right)^{\frac12}} \\ & = \frac{e^{-\frac{\beta}{2(\beta+1)}t^2}t}{(\beta+1)^{\frac32}},
\end{aligned}\]}
as well as
{ 
\[\begin{aligned}
\mb{E}G'(t) & = \int _{-\infty}^\infty (1-\beta z^2)z^{-\frac{\beta}{2}z^2}\diff\edit{\varphi(z-t)}
\\ 
& = \frac{1}{\sqrt{2\pi}}\int_{-\infty}^\infty \left(1-\beta z^2\right)e^{-\frac{\beta}{2}z^2-\frac{1}{2}(z-t)^2}\diff z\\
& =\frac{e^{-\frac{\beta}{2(\beta+1)}t^2}}{\sqrt{2\pi}}\int_{-\infty}^\infty \left[1-\beta\left(u+\frac{t}{\beta+1}\right)^2\right]e^{-\frac{\beta+1}{2}u^2}\diff u \\ 
& = \frac{e^{-\frac{\beta}{2(\beta+1)} t^2}}{\sqrt{2\pi}} \left[\left(1-\frac{\beta t^2}{(\beta+1)^2}\right)\frac{\sqrt{\pi}}{\left(\frac{\beta+1}{2}\right)^{\frac12}}-\frac{\beta\sqrt{\pi}}{2\left(\frac{\beta+1}{2}\right)^{\frac32}} \right]\\
&= \frac{e^{-\frac{\beta}{ 2(\beta+1)}t^2}}{(\beta+1)^{\frac52}}\left((\beta+1)^2-\beta t^2-\beta(1+\beta)\right)\\
&= \frac{e^{-\frac{\beta}{ 2(\beta+1)}t^2}}{(\beta+1)^{\frac52}}\left(1+\beta-\beta t^2\right).
\end{aligned}\]
}
From these computations, and the remark after \cref{lem:gaussian-int}, we readily obtain the asymptotics of $\mu_t$ as in \cref{lem:moments-p} upon multiplying by $\sqrt{n}$.

We now compute $\Sigma_t$. Completing the square gives
\[\edit{\beta} z^2+\frac{1}{2}(z-t)^2 =\frac{2\beta +1}{2}u^2+\frac{\beta t^2}{2\beta +1}\quad \text{where}\quad u = z-\frac{t}{2\beta+1}.\] 
Hence using \cref{lem:gaussian-int} we compute
{ 
\[ 
\begin{aligned}
\mb{E}G^2(t) & = \frac{1}{\sqrt{2\pi}}\int _{-\infty}^\infty z^2e^{-\beta z^2-\frac12(z-t)^2}\diff z \\ 
& = \frac{e^{-\frac{\beta}{(2\beta+1)}t^2}}{\sqrt{2\pi}}\int_{-\infty}^\infty \left(u+\frac{t}{2\beta+1}\right)^2e^{-\frac{1+2\beta}{2}u^2}\diff u \\ 
& = \frac{e^{-\frac{\beta}{(2\beta+1)}t^2}}{\sqrt{2\pi}}  \left[\left(\frac{t}{2\beta+1}\right)^2 \frac{\sqrt{\pi}}{\left(\frac{2\beta+1}{2}\right)^{\frac12}}+\frac{\sqrt{\pi}}{2\left(\frac{2\beta+1}{2}\right)^{\frac32}} \right] \\
& = \frac{e^{-\frac{\beta}{(2\beta+1)}t^2}}{(2\beta +1)^{\frac52}} \left(t^2+2\beta+1\right),
\end{aligned}
\]
}
as well as
{ 
\[
\begin{aligned}
\mb{E}\left[G(t)G'(t)\right] & = \frac{1}{\sqrt{2\pi}}\int _{-\infty}^\infty z(1-\beta z^2)e^{-\beta z^2-\frac12(z-t)^2}\diff z \\ 
& = \frac{e^{-\frac{\beta}{(2\beta+1)}t^2}}{\sqrt{2\pi}}\int_{-\infty}^\infty \left(u+\frac{t}{2\beta+1}-\beta\left(u+\frac{t}{2\beta+1}\right)^3\right)e^{-\frac{1+2\beta}{2}u^2}\diff u \\ 
& = \frac{e^{-\frac{\beta}{2\beta+1}t^2}}{\sqrt{2\pi}}  \left[\left(\frac{t}{2\beta+1}-\beta\left(\frac{t}{2\beta+1}\right)^3\right) \frac{\sqrt{\pi}}{\left(\frac{2\beta+1}{2}\right)^{\frac12}}-\left(\frac{3t\beta}{2\beta+1}\right)\frac{\sqrt{\pi}}{2\left(\frac{2\beta+1}{2}\right)^{\frac32}} \right] \\
& = \frac{e^{-\frac{\beta}{2\beta+1}t^2}}{(2\beta +1)^{\frac72}} \left[t(2\beta+1)^2 -\beta t^3 -3t\beta(2\beta+1)\right]\\
& = \frac{e^{-\frac{\beta}{2\beta+1}t^2}}{(2\beta +1)^{\frac72}} \left(-2\beta^2t+\beta t-\beta t^3+t\right),
\end{aligned} 
\]}
and, finally,
{ 
\[ 
\begin{aligned}
\mb{E}G'^2(t) & = \frac{1}{\sqrt{2\pi}}\int _{-\infty}^\infty (1-\beta z^2)^2e^{-\beta z^2-\frac12(z-t)^2}\diff z \\ 
& = \frac{e^{-\frac{\beta}{2\beta+1}t^2}}{\sqrt{2\pi}}\int_{-\infty}^\infty \left(1-\beta\left(u+\frac{t}{2\beta+1}\right)^2\right)^2e^{-\frac{1+2\beta}{2}u^2}\diff u \\ 
& = \frac{e^{-\frac{\beta}{2\beta+1}t^2}}{\sqrt{2\pi}}  \Bigg[\left(1-\frac{\beta t^2}{(2\beta+1)^2}\right)^2\frac{\sqrt{\pi}}{\left(\frac{2\beta+1}{2}\right)^{\frac12}} \\
&\hspace{2.75cm}+ \left(\frac{6\beta^2t^2}{(2\beta+1)^2}-2\beta\right) \frac{\sqrt{\pi}}{2\left(\frac{2\beta+1}{2}\right)^{\frac32}}+\beta^2\cdot \frac{3\sqrt{\pi}}{4\left(\frac{2\beta+1}{2}\right)^{\frac52}} \Bigg] \\
& = \frac{e^{-\frac{\beta}{2\beta+1}t^2}}{(2\beta +1)^{\frac92}} \Big[((2\beta+1)^2-\beta t^2)^2+(2\beta+1)6\beta^2t^2-2\beta(2\beta+1)^3+3\beta^2(2\beta+1)^2\Big]\\
& = \frac{e^{-\frac{\beta}{2\beta+1}t^2}}{(2\beta +1)^{\frac92}} \left(12\beta^4+4\beta^3(t^2+5)+\beta^2(t^4-2t^2+15)-2\beta (t^2-3)+1 \right). 
\end{aligned}\]
}
It is easy to check that entries of $\Sigma_t$ are asymptotically the corresponding second moments. Together, we readily obtain the asymptotics of $\Sigma_t$ as indicated in \cref{lem:moments-p}.\qed

\subsection{\texorpdfstring{Proof of \Cref{lem:eta}}{Proof of}}

In this section, we prove \cref{lem:eta} on cumulants of $Y = \Sigma_t^{-\frac12}\left(G-\mb{E}G, G'-\mb{E}G'\right)$. To upper bound, we do not need to track the leading coefficients to ensure that they do not vanish when we combine applications of \cref{lem:gaussian-int}. 
\edit{First, as $s$ is a constant (we only apply $s=2, 3$), cumulants of order $s$ are clearly $O(\eta_{s})$.}
To bound $\eta_{s}$, we recall $\Sigma_t^{-1}$ from the proof of \cref{lem:phi-t} and apply H\"{o}lder's inequality:
\edit{\begin{align*}
\eta_s & = \mb{E} \left[\Vert Y\Vert ^s\right]
\\ & \le \mb{E}\left\Vert\left(G-\mb{E}G, G'-\mb{E}G'\right)^\intercal \Sigma_t^{-1}\left(G-\mb{E}G, G'-\mb{E}G'\right)\right\Vert^{\frac{s}{2}}
\\ & \lesssim \beta^{\frac{s}{4}} e^{\frac{st^2}{4}} \mb{E}\left|\beta (G-\mb{E}G)^2+2t(G-\mb{E}G)(G'-\mb{E}G')+(G'-\mb{E}G')^2\right|^{\frac{s}{2}}
\\ & \lesssim \beta^{\frac{s}{4}} e^{\frac{st^2}{4}} \left(\beta^{\frac{s}{2}}\mb{E}|G|^s+\mb{E}|G'|^s\right)
\\ &\lesssim \beta^{\frac{s}{4}} e^{\frac{(s-2)t^2}{4}} \int_{0}^\infty h_s\left(|z|\right)e^{-\frac{3\beta+1}{2}\left(z-\frac{t}{3\beta+1}\right)^2}\diff z,
\end{align*}}
\edit{where $h_s(u)\coloneqq\beta^{\frac{s}{2}}u^s+(1+\beta u^2)^s$.
Note that the shift $t/(3\beta+1)\ll ((3\beta+1)/2)^{-1/2}$, so 
by linearity of integration, we may apply \cref{lem:gaussian-int} to bound the integral of each monomial $|z|^\ell$ by $O(\beta^{-(\ell+1)/2})$. By monotonicity of $h$ on $\mb{R}_{\ge 0}$ and since $t\in T$, for some constant $C>0$,
\[\eta_s  \lesssim \beta^{\frac{s-2}{4}} e^{\frac{(s-2)t^2}{4}} h\left(C\beta^{-\frac12}\right)
\lesssim  \beta^{\frac{s-2}{4}} e^{\frac{(s-2)t^2}{4}}
\]
upon noting $u\mapsto h(u/\sqrt{\beta})$ has constant coefficients. This proves \cref{lem:eta}.}\qed

\subsection{\texorpdfstring{Proof of \Cref{lem:error-higher}}{Proof of}}
\label{sec:pf-error-higher}
\edit{
Fix $n, \beta$ sufficiently large as well as $t$. Define for a multi-index $\alpha$ that $h(\mbf{x})=\mbf{x}^\alpha g_s(\mbf{x})$. Then, for $s=3$
\[
\mathscr{F}h(\mbf{z}) =\partial ^\alpha \left(\mathscr{F}(q_t)-2\pi\varphi-\sum_{k=0}^3 \frac{\varphi}{k!(3-k)!\sqrt{n}}\kappa_t^{(k, 3-k)} H^{(k, 3-k)}
 \right)(\mbf{z}),
\]
and for $s=2$ we have the same statement with the last summand omitted. Note that we use 
\begin{equation*}
    \mathscr{F}h(\mbf{z}) \coloneqq \int_{\mb{R}^2} e^{-\mathrm{i}\langle \mbf{x},\mbf{z}\rangle}h(\mbf{x})\diff \mbf{x}
\end{equation*}
to denote the Fourier transform of $h$. We also omit the dependence of $h$ on $s$ and $\alpha$ for brevity.
By Fourier inversion, it suffices to show that for any multi-index $\alpha$ with order $|\alpha|\le s$ that
\begin{equation}
\label{eq:higher-error-goal}
\left\vert h(\mbf{x})\right| = \left|\frac{1}{(2\pi)^{2}}\int_{\mb{R}^2} e^{-\mathrm{i}\langle \mbf{z}, \mbf{x}\rangle}\mathscr{F}{h}(\mbf{z})\diff \mbf{z}\right|\lesssim \int_{\mb{R}^2} \left|\mathscr{F}{h}(\mbf{z})\right|\diff \mbf{z} \lesssim n^{-\frac{s-1}{2}}\eta_{s+1}
\end{equation}
We apply \cite[Theorem 9.10]{bhattacharya2010normal}---which is not asymptotic and has explicit constants in $s$ only---so we may use it even though $q_t$ depends on $\beta$ to obtain that
\begin{equation}
\label{eq:br-9.10}
\left|\mathscr{F}{h}(\mbf{z})\right| \lesssim  n^{-\frac{s-1}{2}}\eta_{s+1} \Vert \mbf{z}\Vert^{O(1)}e^{-\frac{\Vert \mbf{z}\Vert^2}{4}}
\end{equation}
provided $\Vert \mbf{z}\Vert\le  a\sqrt{n}$ for some $a\asymp\eta_{s+1}^{-\frac{1}{s-1}}$. 
By \cref{lem:eta}, we have that
\begin{equation}
\label{eq:small-z}
\int_{\Vert \mbf{z}\Vert \le a\sqrt{n}}\left|\mathscr{F}{h}(\mbf{z})\right| \diff\mbf{z}\lesssim  n^{-\frac{s-1}{2}}\eta_{s+1}\int_{\mb{R}^2}\Vert \mbf{z}\Vert^{O(1)}e^{-\frac{\Vert\mbf{z}\Vert^2}{4}} \diff \mbf{z} \lesssim n^{-\frac{s-1}{2}}\eta_{s+1}.
\end{equation}
Recall that $q_t$ is the density of $n^{-\frac12}\sum_{i=1}^n Y_i$. Let $f$ denote the density of $W\coloneqq\sqrt{5}(Y_1+\dots +Y_5)$ which exists and is bounded by \cref{lem: pt.bdd}. Hence, $\msc{F}f\in L^1(\mb{R}^2)$ and 
\[ \varepsilon \coloneqq \sup_{\Vert \mbf{z}\Vert >a} \left|\mathscr{F}{f}(\mbf{z})\right| <1.
\]
Now, $q_t$ is the density of $\sqrt{n/5}$ times the sum of $n/5$ many i.i.d. copies of $W$, so by properties of the Fourier transform and the product rule,
\begin{equation}
\label{eq:big-z-exp}
\begin{aligned}
\int_{\Vert \mbf{z}\Vert > a\sqrt{n}} \left|\partial^\alpha \mathscr{F}{q_t}(\mbf{z})\right|\diff\mbf{z} & \lesssim \eta_{|\alpha|}n^{\frac{|\alpha|}{2}}\varepsilon^{n/5-|\alpha|-1}\int_{\mb{R}^2}\left|\mathscr{F}{f}\left(\frac{\mbf{z}}{\sqrt{n}}\right)\right| \diff\mbf{z}
\\ & \lesssim \left(n\beta e^{\frac{t^2}{2}}\right)^{O(1)} \varepsilon^{n/5-s-1}
\\ & \ll n^{-\frac{s-1}{2}}\eta_{s+1}
\end{aligned}
\end{equation}
for sufficiently large $n$. Finally, we bound similar to \cref{lem:error-3}:
\begin{equation}
\label{eq:big-phi-poly}
\int_{\Vert \mbf{z}\Vert > a\sqrt{n}} \left|\partial^\alpha\right|(\mbf{z})\diff\mbf{z}
 \lesssim \int_{{\Vert \mbf{z}\Vert > a\sqrt{n}}} \Vert \mbf{z}\Vert ^{O(1)}e^{-\frac{\Vert \mbf{z}\Vert^2}{2}}\diff\mbf{z}\ll n^{-\frac{s-1}{2}}\eta_{s+1}
\end{equation}
and for the $s=3$ we also have the additional term
\begin{equation}
\label{eq:big-z-poly}
\begin{aligned}
\int_{\Vert \mbf{z}\Vert > a\sqrt{n}} &\left|\partial^\alpha \sum_{k=0}^3\frac{\varphi}{k!(3-k)!\sqrt{n}} \kappa_t^{(k, 3-k)} H^{(k, 3-k)}\right|(\mbf{z})\diff\mbf{z}
\\
& \lesssim n^{-\frac12}\sum_{k=0}^3\kappa_{t}^{(k, 3-k)}\int_{{\Vert \mbf{z}\Vert > a\sqrt{n}}}\left|\partial^{\alpha}H^{(k, 3-k)}\varphi\right|(\mbf{z}) \diff\mbf{z}
\\ & \lesssim n^{-\frac12}\eta_3 \int_{{\Vert \mbf{z}\Vert > a\sqrt{n}}} \Vert \mbf{z}\Vert ^{O(1)}e^{-\frac{\Vert \mbf{z}\Vert^2}{2}}\diff\mbf{z}
\\ & \ll n^{-\frac{s-1}{2}}\eta_{s+1}
\end{aligned}
\end{equation}
Now, in the last step of both \cref{eq:big-phi-poly,eq:big-z-poly}, we use that the standard Gaussian integral outside the ball at the origin converges to zero exponentially quickly as radius $a\sqrt{n}\to\infty$ by \cref{eq:rate}, so in particular
\[ 
\int_{{\Vert \mbf{z}\Vert > a\sqrt{n}}} \Vert \mbf{z}\Vert ^{O(1)}e^{-\frac{\Vert \mbf{z}\Vert^2}{2}}\diff\mbf{z} \ll (a\sqrt{n})^{-(s-1)} \asymp n^{-\frac{s-1}{2}}\eta_{s+1}
\]
Combining \cref{eq:small-z,eq:big-z-exp,eq:big-phi-poly,eq:big-z-poly} proves \cref{eq:higher-error-goal} and hence \cref{lem:error-higher} for both $s=2$ and $s=3$ cases. \qed}

\subsection{\texorpdfstring{Proof of \Cref{lem: pt.bdd}}{Proof of}
} \label{sec: proof.pt.bdd}


Point {\it 1} in \Cref{thm:kac-rice} can readily be seen to hold because of the explicit form of both of the fields. 
Point {\it 4} also readily holds, since $F_n''$ is a Lipschitz function for every realization of $X_i$, as a sum of Lipschitz functions.
We focus on showing Point {\it 3}, the proof of which can be repeated essentially verbatim to deduce Point {\it 2}. 

\subsubsection*{Proof of Point {\it 3}}
Observe that $\upmu_t=\upnu_t^{\ast n}$, where $\upnu_t$ is the law of 
\begin{equation*} \label{eq: Gt-prime}
    \begin{bmatrix}
        G(t)\\ G'(t)
    \end{bmatrix} = \begin{bmatrix}
        g(Z)\\
        g'(Z)
    \end{bmatrix}
\end{equation*}
with $Z\sim N(t,1)$ and $g(z) = ze^{-{\beta} z^2/2}$.  
(Also, for $n=1$ we have $\upmu_t=\upnu_t$, and $\upnu_t$ cannot have a continuous density on $\mb{R}^2$, since both components of a drawn random vector $(G(t), G'(t))$ are functions of the same one-dimensional Gaussian random variable.) 

\edit{
We first show that $\mathscr{F}(\upnu_t^{\ast n}) = (\mathscr{F}\upnu_t)^n\in L^1(\mb{R}^2)$. We work with a fixed $t$ and, by translation invariance of the standard Gaussian, we may take $t=0$ without loss of generality.
Proving this would imply that $\upmu_t$ has a density $p_t\in\mathscr{C}^0(\mb{R}^2)$ satisfying $p_t(\mbf{x})\to0$ as $\|\mbf{x}\|\to\infty$ by virtue of Fourier inversion and the Riemann-Lebesgue lemma. 
We also perform computations as if $\upnu_t^{\ast n}$ were already a function, and all arguments can be justified by appealing to the framework of Schwarz distributions $\mathcal{S}'(\mb{R}^2)$ and duality. 

We write $\xi=(\xi_1,\xi_2)$ with $\|\xi\|=\rho$, and  $\omega=(\cos\theta,\sin\theta)$ so that $\xi=\rho\omega$. We then have 
\begin{equation*}
    \mathscr{F}\upnu_t(\xi)=\frac{1}{\sqrt{2\pi}}\int_{\mb{R}} e^{-\mathrm{i}\rho\phi_\theta(x)}e^{-\frac{x^2}{2}}\diff x
\end{equation*}
with phase $\phi_\theta(x)=\cos\theta g(x)+\sin\theta g'(x)$. We will show the uniform bound 
\begin{equation} \label{eq:uniform-decay}
\left|\mathscr F\nu_t(\xi)\right|\lesssim \frac{1}{\sqrt{1+\|\xi\|}},
\end{equation}
for all $\xi\in\mb{R}^2$, where the implicit constant depends only on $\beta$. To this end, we compute 
\begin{align*}
    g'(x)&=(1-\beta x^2)e^{-\beta\frac{x^2}{2}},\\
    g''(x)&=\beta x(\beta x^2-3)e^{-\beta\frac{x^2}{2}},\\
    g'''(x)&=\beta(-\beta^2x^4+6\beta x^2-3)e^{-\beta\frac{x^2}{2}}.
\end{align*}
Set $\psi(x)=(g(x), g'(x))$. We have 
\begin{equation*}
    \det (\psi'(x),\psi''(x)) = g'(x)g'''(x)-(g''(x))^2 = -\beta(\beta^2 x^4+3)e^{-\beta x^2}<0
\end{equation*}
for all $x.$ In particular the determinant never vanishes.
Observe that $\phi_\theta'(x)=\cos\theta g'(x)+\sin\theta g''(x)$, and any stationary point $x_0$ of $\phi_\theta$ satisfies $g'(x_0)=0$. If also $\phi''(x_0)=\cos\theta g''(x_0)+\sin\theta g'''(x_0)=0$, then $(\cos\theta, \sin\theta)$ would be a nontrivial vector orthogonal to both $(g'(x_0), g''(x_0))$ and $(g''(x_0), g'''(x_0))$, forcing $g'(x_0)g'''(x_0)-(g''(x_0))^2=0$, a contradiction. Hence all stationary points of $\phi_\theta$ are non-degenerate, uniformly in $\theta$.

Fix $R>0$ and a smooth cutoff function $\chi\in C_c^\infty(\mb{R})$ with $\chi\equiv1$ on $[-1,1]$, supported within $[-2,2]$. Then set $\chi_R(x)=\chi(x/R)$.
We write 
\begin{align*}
    \mathscr{F}\nu_t(\xi)&=\frac{1}{\sqrt{2\pi}}\int_{\mb{R}} e^{-\mathrm{i}\rho\phi_\theta(x)}\chi_R(x)e^{-\frac{x^2}{2}}\diff x + \frac{1}{\sqrt{2\pi}}\int_{\mb{R}} e^{-\mathrm{i}\rho\phi_\theta(x)}(1-\chi_R(x))e^{-\frac{x^2}{2}}\diff x \\
    &= I_1 + I_2.
\end{align*}
We have 
\begin{equation*}
    |I_2|\leq \int_{|x|\geq R} e^{-\frac{x^2}{2}}\leq \frac{1}{\sqrt{2\pi}}\frac{2}{R}e^{-\frac{R^2}{2}},
\end{equation*}
which, as $R$ is fixed, is smaller than $(1+\rho)^{-\frac12}$ whenever $\rho$ is large enough, uniformly in $\theta$. Concerning $I_1$, because $\phi'_\theta$ is analytic it only has finitely many zeros $x_1(\theta), \ldots, x_m(\theta)$ on $[-2R,2R]$ (with $m\leq 3$). They are all non-degenerate: $\phi''_\theta(x_j(\theta))\neq0$. Therefore there exist disjoint intervals $J_j(\theta)\subset[-2R,2R]$ around each $x_j(\theta)$ of radius $\delta>0$, and some $c_*=c_*(R,\beta,\delta)>0$ such that $|\phi''_\theta(x)|\geq c_*$ for all $x$ in the union of the intervals $J_j(\theta)$ and all $\theta\in\mathbb{S}^1$. 
On each $J_j(\theta)$, we may apply the method of stationary phase to find 
\begin{align*}
    \left|\int_{\cup_j J_j(\theta)} e^{-\mathrm{i}\rho\phi_\theta(x)}\chi_R(x)e^{-\frac{x^2}{2}}\diff x\right|&\leq \frac{C}{\sqrt{\rho}}\left(\left\|\chi_R(\cdot) e^{-\frac{(\cdot)^2}{2}}\right\|_{L^\infty(\mb{R})}+\left\|\left(\chi_R(\cdot) e^{-\frac{(\cdot)^2}{2}}\right)'\right\|_{L^\infty(\mb{R})}\right)\\
    &\leq \frac{C}{\sqrt{\rho}}\left(1+e^{-\frac12}+\frac{\|\chi'\|_{L^\infty(\mb{R})}}{R}\right),
\end{align*}
where $C=C(R,\beta,\delta)>0$ is independent of $\rho,\theta$. Now set $W(\theta)\coloneqq [-2R,2R]\setminus \cup_j J_j(\theta)$. 
By compactness and continuity, we again have $c_W\coloneqq \inf_{\theta\in\mb{S}^1}\inf_{x\in W(\theta)}|\phi_\theta'(x)|>0$. We now look to use the method of non-stationary phase: integration by parts gives 
\begin{align*}
    \left|\int_{W(\theta)} e^{-\mathrm{i}\rho\phi_\theta(x)}\chi_R(x)e^{-\frac{x^2}{2}}\diff x\right|\leq \frac{1}{\rho}\int_{W(\theta)} \left|\frac{\left(\chi_R e^{-\frac{(\cdot^2)}{2}}\right)'(x)}{\phi_\theta'(x)}-\frac{\chi_R(x)e^{-\frac{x^2}{2}}\phi_\theta''(x)}{(\phi_\theta'(x))^2}\right|\diff x\\
    \leq \frac{1}{\rho}\left(\frac{1}{c_W}\left\|\left(\chi_Re^{-\frac{(\cdot)^2}{2}}\right)'\right\|_{L^1(\mb{R})}+\frac{M}{c_W^2}\left\|\chi_R e^{-\frac{(\cdot)^2}{2}}\right\|_{L^1(\mb{R})}\right),
\end{align*}
where $M=\sup_{x\in [-2R,2R]}\sup_{ \theta\in\mathbb{S}^1}|\phi_\theta''(x)|<\infty$.
All in all, both bounds yield $|\mathscr{F}\nu_t(\xi)|\lesssim \|\xi\|^{-\frac12}$ for  $\rho\coloneqq\|\xi\|\geq1$ sufficiently large, which proves \eqref{eq:uniform-decay}.

We now have 
\begin{equation*}
\int_{\mb{R}^2}|\mathscr{F}\nu_t(\xi)|^n \diff \xi \lesssim \int_{\|\xi\|\leq 1} 1\diff\xi + \int_{\|\xi\|>1} |\xi|^{-\frac{n}{2}}\diff\xi,
\end{equation*}
which is finite as long as $n>4$. By Fourier inversion, we have the desired conclusion.
}

To deduce that $(t,\mbf{x})\mapsto p_t(\mbf{x})$ is continuous on $T\times\mb{R}^2$, we note that 
\begin{equation*}
    p_t(\mbf{x}) = \frac{1}{(2\pi)^2}\int_{\mb{R}^2} e^{\mathrm{i}\langle\mbf{x}, \mbf{z}\rangle}\int_{\mb{R}^n} \exp\left(-\mathrm{i}\left\langle\mbf{z}, \begin{bmatrix} \sum_{j=1}^n g(\xi_j)\\ \sum_{j=1}^n g'(\xi_j)\end{bmatrix}\right\rangle\right) \gamma_t(\xi_1)\cdots\gamma_t(\xi_n)\diff\xi\diff\mbf{z}.
\end{equation*}
We can conclude by the Lebesgue dominated convergence theorem.

\section{Concluding remarks}

We showed that the expected number of modes of a Gaussian KDE with bandwidth $\beta^{-\frac12}$ of $n\ge1$ samples drawn iid from $N(0, 1)$ is of order $\Theta(\sqrt{\beta\log \beta})$ for $n^c\lesssim \beta\lesssim n^{2-c}$, where $c>0$ is arbitrarily small. We also provide a precise picture of where the modes are located. 

\edit{The with high probability version of the statements and the question in the higher-dimensional case remains open: we conjecture the number of modes to be $\Theta(\sqrt{\beta^d\log \beta})$. We also raise the question for non-Gaussian densities as well as the case of the unit sphere $\mathbb{S}^{d-1}$ with uniformly distributed samples.}

\subsubsection*{Acknowledgments}
The authors would like to thank Enno Mammen for useful discussion and sharing important references. We also thank Dan Mikulincer for discussions on Gaussian approximation using Edgeworth expansions, Valeria Banica for comments on the method of stationary phase, and Alexander Zimin for providing \Cref{fig: 1}. We finally thank all the reviewers for their comments, which have greatly improved the quality of the paper.

\subsubsection*{Funding}
B.G. was supported by a Sorbonne Emergences grant, and a gift from Google.
P.R. was supported by NSF grants DMS-2022448, CCF-2106377, and a gift from Apple.
Y.S. was supported by the MIT UROP and MISTI France Programs.



\bibliographystyle{alpha}
\bibliography{refs}{}

@book{adler2009random,
  title={Random fields and geometry},
  author={Adler, Robert J and Taylor, Jonathan E},
  year={2009},
  publisher={Springer Science \& Business Media}
}

@inproceedings{carreira2003number,
  title={On the number of modes of a {G}aussian mixture},
  author={Carreira-Perpin{\'a}n, Miguel A and Williams, Christopher KI},
  booktitle={International Conference on Scale-Space Theories in Computer Vision},
  pages={625--640},
  year={2003},
  organization={Springer}
}

@book{bhattacharya2010normal,
  title={Normal approximation and asymptotic expansions},
  author={Bhattacharya, Rabi N and Rao, R Ranga},
  year={2010},
  publisher={SIAM}
}

@article {mammen97,
    AUTHOR = {Konakov, V. and Mammen, E.},
     TITLE = {The shape of kernel density estimates in higher dimensions},
   JOURNAL = {Math. Methods Statist.},
  FJOURNAL = {Mathematical Methods of Statistics},
    VOLUME = {6},
      YEAR = {1997},
    NUMBER = {4},
     PAGES = {440--464 (1998)},
      ISSN = {1066-5307},
   MRCLASS = {62G05},
  MRNUMBER = {1621247},
MRREVIEWER = {Serguei Yu. Novak},
}

@book{Tsy09,
	address = {New York},
	author = {Tsybakov, Alexandre B.},
	publisher = {Springer},
	series = {Springer Series in Statistics},
	title = {Introduction to nonparametric estimation},
	year = {2009}}

@article {mammen91,
    AUTHOR = {Mammen, E. and Marron, J. S. and Fisher, N. I.},
     TITLE = {Some asymptotics for multimodality tests based on kernel
              density estimates},
   JOURNAL = {Probab. Theory Related Fields},
  FJOURNAL = {Probability Theory and Related Fields},
    VOLUME = {91},
      YEAR = {1992},
    NUMBER = {1},
     PAGES = {115--132},
      ISSN = {0178-8051},
   MRCLASS = {62G10 (62E20)},
  MRNUMBER = {1142765},
MRREVIEWER = {Margaret P. Gessaman},
       DOI = {10.1007/BF01194493},
       URL = {https://doi.org/10.1007/BF01194493},
}

@article {mammen95,
    AUTHOR = {Mammen, E.},
     TITLE = {On qualitative smoothness of kernel density estimates},
   JOURNAL = {Statistics},
  FJOURNAL = {Statistics. A Journal of Theoretical and Applied Statistics},
    VOLUME = {26},
      YEAR = {1995},
    NUMBER = {3},
     PAGES = {253--267},
      ISSN = {0233-1888},
   MRCLASS = {62G05},
  MRNUMBER = {1365677},
MRREVIEWER = {Theo Gasser},
       DOI = {10.1080/02331889508802494},
       URL = {https://doi.org/10.1080/02331889508802494},
}

@book {devgyo,
    AUTHOR = {Devroye, Luc and Gy\"{o}rfi, L\'{a}szl\'{o}},
     TITLE = {Nonparametric density estimation},
    SERIES = {Wiley Series in Probability and Mathematical Statistics:
              Tracts on Probability and Statistics},
      NOTE = {The $L_1$ view},
 PUBLISHER = {John Wiley \& Sons, Inc., New York},
      YEAR = {1985},
     PAGES = {xi+356},
      ISBN = {0-471-81646-9},
   MRCLASS = {62G05 (60F05 62H30)},
  MRNUMBER = {780746},
MRREVIEWER = {Paul Deheuvels},
}

@article{carreira2015review,
  title={A review of mean-shift algorithms for clustering},
  author={Carreira-Perpin{\'a}n, Miguel A},
  journal={arXiv preprint arXiv:1503.00687},
  year={2015}
}

@article{geshkovski2024emergence,
  title={The emergence of clusters in self-attention dynamics},
  author={Geshkovski, Borjan and Letrouit, Cyril and Polyanskiy, Yury and Rigollet, Philippe},
  journal={Advances in Neural Information Processing Systems},
  volume={36},
  year={2024}
}

@article{geshkovski2023mathematical,
  title={A mathematical perspective on transformers},
  author={Geshkovski, Borjan and Letrouit, Cyril and Polyanskiy, Yury and Rigollet, Philippe},
  journal={Bulletin of the American Mathematical Society},
  volume={62},
  number={3},
  pages={427--479},
  year={2025}
}

@inproceedings{sander2022sinkformers,
  title={Sinkformers: Transformers with doubly stochastic attention},
  author={Sander, Michael E and Ablin, Pierre and Blondel, Mathieu and Peyr{\'e}, Gabriel},
  booktitle={International Conference on Artificial Intelligence and Statistics},
  pages={3515--3530},
  year={2022},
  organization={PMLR}
}

@article{criscitiello2024synchronization,
  title={Synchronization on circles and spheres with nonlinear interactions},
  author={Criscitiello, Christopher and Rebjock, Quentin and McRae, Andrew D and Boumal, Nicolas},
  journal={arXiv preprint arXiv:2405.18273},
  year={2024}
}

@inproceedings{vaswani2017attention,
  author = {Vaswani, Ashish and Shazeer, Noam and Parmar, Niki and Uszkoreit, Jakob and Jones, Llion and Gomez, Aidan N and Kaiser, Lukasz and Polosukhin, Illia},
 booktitle = {Advances in Neural Information Processing Systems},
 title = {Attention is All you Need},
 volume = {30},
 year = {2017}
}

@article{geshkovski2024dynamic,
  title={Dynamic metastability in the self-attention model},
  author={Geshkovski, Borjan and Koubbi, Hugo and Polyanskiy, Yury and Rigollet, Philippe},
  journal={arXiv preprint arXiv:2410.06833},
  year={2024}
}

@article{bally2019non,
  title={Non universality for the variance of the number of real roots of random trigonometric polynomials},
  author={Bally, Vlad and Caramellino, Lucia and Poly, Guillaume},
  journal={Probability Theory and Related Fields},
  volume={174},
  pages={887--927},
  year={2019},
  publisher={Springer}
}

@article{fukunaga1975estimation,
  title={The estimation of the gradient of a density function, with applications in pattern recognition},
  author={Fukunaga, Keinosuke and Hostetler, Larry},
  journal={IEEE Transactions on information theory},
  volume={21},
  number={1},
  pages={32--40},
  year={1975},
  publisher={IEEE}
}

@article{cheng1995mean,
  title={Mean shift, mode seeking, and clustering},
  author={Cheng, Yizong},
  journal={IEEE transactions on pattern analysis and machine intelligence},
  volume={17},
  number={8},
  pages={790--799},
  year={1995},
  publisher={IEEE}
}

@article{rodriguez2014clustering,
  title={Clustering by fast search and find of density peaks},
  author={Rodriguez, Alex and Laio, Alessandro},
  journal={Science},
  volume={344},
  number={6191},
  pages={1492--1496},
  year={2014},
  publisher={American Association for the Advancement of Science}
}

@article{comaniciu2002mean,
  title={Mean shift: A robust approach toward feature space analysis},
  author={Comaniciu, Dorin and Meer, Peter},
  journal={IEEE Transactions on pattern analysis and machine intelligence},
  volume={24},
  number={5},
  pages={603--619},
  year={2002},
  publisher={IEEE}
}

@inproceedings{
bruno2024emergence,
title={Emergence of meta-stable clustering in mean-field transformer models},
author={Giuseppe Bruno and Federico Pasqualotto and Andrea Agazzi},
booktitle={The Thirteenth International Conference on Learning Representations},
year={2025},
url={https://openreview.net/forum?id=eBS3dQQ8GV}
}

@article{geshkovski2024measure,
  title={Measure-to-measure interpolation using Transformers},
  author={Geshkovski, Borjan and Rigollet, Philippe and Ruiz-Balet, Dom{\`e}nec},
  journal={arXiv preprint arXiv:2411.04551},
  year={2024}
}

@article{grenier2000nonlinear,
  title={On the nonlinear instability of {Euler and Prandtl} equations},
  author={Grenier, Emmanuel},
  journal={Communications on Pure and Applied Mathematics},
  volume={53},
  number={9},
  pages={1067--1091},
  year={2000},
  publisher={Wiley Online Library}
}

@book{azais2009level,
  title={Level sets and extrema of random processes and fields},
  author={Aza{\"\i}s, Jean-Marc and Wschebor, Mario},
  year={2009},
  publisher={John Wiley \& Sons}
}

@article{carreira2000mode,
  title={Mode-finding for mixtures of {G}aussian distributions},
  author={Carreira-Perpin{\'a}n, Miguel A},
  journal={IEEE Transactions on Pattern Analysis and Machine Intelligence},
  volume={22},
  number={11},
  pages={1318--1323},
  year={2000},
  publisher={IEEE}
}

@article{carreira2007gaussian,
  title={{Gaussian mean-shift is an EM algorithm}},
  author={Carreira-Perpin{\'a}n, Miguel A},
  journal={IEEE Transactions on Pattern Analysis and Machine Intelligence},
  volume={29},
  number={5},
  pages={767--776},
  year={2007},
  publisher={IEEE}
}

@article{fan2021tap,
author = {Zhou Fan and Song Mei and Andrea Montanari},
title = {{TAP free energy, spin glasses and variational inference}},
volume = {49},
journal = {The Annals of Probability},
number = {1},
publisher = {Institute of Mathematical Statistics},
pages = {1 -- 45},
year = {2021},
doi = {10.1214/20-AOP1443},
URL = {https://doi.org/10.1214/20-AOP1443}
}

@inproceedings{maillard2020landscape,
  title={Landscape complexity for the empirical risk of generalized linear models},
  author={Maillard, Antoine and Ben Arous, G{\'e}rard and Biroli, Giulio},
  booktitle={Mathematical and Scientific Machine Learning},
  pages={287--327},
  year={2020},
  organization={PMLR}
}

@article{alcalde2025attention,
  title={Attention's forward pass and {Frank-Wolfe}},
  author={Alcalde, Albert and Geshkovski, Borjan and Ruiz-Balet, Dom{\`e}nec},
  journal={arXiv preprint arXiv:2508.09628},
  year={2025}
}

@article{polyanskiy2025synchronization,
  title={Synchronization of mean-field models on the circle},
  author={Polyanskiy, Yury and Rigollet, Philippe and Yao, Andrew},
  journal={arXiv preprint arXiv:2507.22857},
  year={2025}
}

@article{karagodin2024clustering,
  title={Clustering in causal attention masking},
  author={Karagodin, Nikita and Polyanskiy, Yury and Rigollet, Philippe},
  journal={Advances in Neural Information Processing Systems},
  volume={37},
  pages={115652--115681},
  year={2024}
}

@article{bruno2025multiscale,
  title={A multiscale analysis of mean-field transformers in the moderate interaction regime},
  author={Bruno, Giuseppe and Pasqualotto, Federico and Agazzi, Andrea},
  journal={arXiv preprint arXiv:2509.25040},
  year={2025}
}

@article{auffinger2013random,
  title={Random matrices and complexity of spin glasses},
  author={Auffinger, Antonio and Ben Arous, G{\'e}rard and {\v{C}}ern{\`y}, Ji{\v{r}}{\'\i}},
  journal={Communications on Pure and Applied Mathematics},
  volume={66},
  number={2},
  pages={165--201},
  year={2013},
  publisher={Wiley Online Library}
}

    \bigskip

\begin{minipage}[t]{.5\textwidth}
{\footnotesize{\bf Borjan Geshkovski}\par
  Inria \&
  Laboratoire Jacques-Louis Lions\par
  Sorbonne Université\par
  4 Place Jussieu\par
  75005 Paris, France\par
 \par
  e-mail: \href{mailto:borjan.geshkovski@inria.fr}{\textcolor{blue}{\scriptsize borjan.geshkovski@inria.fr}}
  }
\end{minipage}
\begin{minipage}[t]{.5\textwidth}
  {\footnotesize{\bf Philippe Rigollet}\par
  Department of Mathematics\par
  Massachusetts Institute of Technology\par
  77 Massachusetts Ave\par
  Cambridge 02139 MA, United States\par
 \par
  e-mail: \href{mailto:blank}{\textcolor{blue}{\scriptsize rigollet@math.mit.edu}}
  }
\end{minipage}%

\begin{center}
\begin{minipage}[t]{.5\textwidth}
  {\footnotesize{\bf Yihang Sun}\par
  Department of Mathematics\par
  Stanford University\par
450 Jane Stanford Way Building 380\par
  Stanford, CA 94305, United States\par
 \par
  e-mail: \href{mailto:blank}{\textcolor{blue}{\scriptsize kimisun@stanford.edu}}
  }
\end{minipage}%
\end{center}

\end{document}